\documentclass[3p]{elsarticle}

\usepackage{graphicx}
\graphicspath{{figures/pics_phase/}}
\usepackage{amsmath,amssymb,amsthm}
\usepackage{mathtools}
\usepackage{enumitem}
\usepackage{hyperref}
\hypersetup{
    colorlinks=true,
    linkcolor=blue,
    urlcolor=red,
    linktoc=all
}
\usepackage[nameinlink,capitalise,noabbrev]{cleveref}
\usepackage[utf8]{inputenc}
\usepackage{float}
\usepackage[british]{babel}
\usepackage{color}
\usepackage{xcolor}
\usepackage{tikz}
\usetikzlibrary{3d}
\usetikzlibrary {arrows.meta} 
\usepackage{tabularx}
\usepackage{lipsum}
\biboptions{numbers,sort&compress}
\usepackage{subcaption}

\def\div{\operatorname{div}}

\def\Th{\mathcal{T}_h}
\def\Itau{\mathcal{I}_\tau}

\def\la{\langle}
\def\ra{\rangle}

\def\Th{\mathcal{T}_h}
\def\Vh{\mathcal{V}_h}
\def\Qh{\mathcal{Q}_h}
\def\Xh{\mathcal{X}_h}
\def\bx{\mathbf{x}}

\def\nn{\nonumber}

\newcommand{\jrho}{[\![\rho]\!]}

\newcommand{\arho}{\left\{\rho\right\}}

\DeclarePairedDelimiter{\norm}{\|}{\|}
\DeclarePairedDelimiter{\snorm}{|}{|}

\def\u{\mathbf{u}}
\def\w{\mathbf{w}}
\def\vv{\mathbf{v}}

\def\phm{\phantom{$-$}}

\newtheorem{lemma}{Lemma}
\newtheorem{problem}[lemma]{Problem}
\newtheorem{theorem}[lemma]{Theorem}

\theoremstyle{definition}

\newtheorem{remark}[lemma]{Remark}

\def\dt{\partial_t}

\def\dtau{d_\tau^{n+1}}

\def\ddt{\frac{\mathrm{d}}{\mathrm{d}t}}

\def\softd{{\leavevmode\setbox1=\hbox{d}%
		\hbox to 1.05\wd1{d\kern-0.4ex{\char039}\hss}}}

\journal{Journal of Computational Physics}

\makeatletter
\def\ps@pprintTitle{%
 \let\@oddhead\@empty
 \let\@evenhead\@empty
 \let\@oddfoot\@empty
 \let\@evenfoot\@empty
}
\makeatother
\begin{document}

\begin{frontmatter}



\title{A simple, fully-discrete, unconditionally energy-stable method for the two-phase Navier-Stokes Cahn-Hilliard model with arbitrary density ratios}


\renewcommand{\thefootnote}{\fnsymbol{footnote}}

\author[1]{A. Brunk}
\ead{abrunk@uni-mainz.de}
\author[2]{M.F.P. ten Eikelder}
\ead{marco.eikelder@tu-darmstadt.de}

\address[1]{Institute of Mathematics, Johannes Gutenberg-University Mainz, Germany}
\address[2]{Institute for Mechanics, Computational Mechanics Group, Technical University of Darmstadt, Germany}

\footnotetext[1]{The authors contributed equally to this work.}

\renewcommand{\thefootnote}{\arabic{footnote}}
\setcounter{footnote}{0}  

\date{\today}


\begin{abstract}

The two-phase Navier-Stokes Cahn-Hilliard (NSCH) mixture model is a key framework for simulating multiphase flows with non-matching densities.
Developing fully discrete, energy-stable schemes for this model remains challenging, due to the possible presence of negative densities.
While various methods have been proposed, ensuring provable energy stability under phase-field modifications, like positive extensions of the density, remains
an open problem. 
We propose a simple, fully discrete, energy-stable method for the NSCH mixture model that ensures stability with respect to the energy functional, 
where the density in the kinetic energy is positively extended.
The method is based on an alternative but equivalent formulation using mass-averaged velocity 
and volume-fraction-based order parameters, simplifying implementation while preserving theoretical consistency.
Numerical results demonstrate that the proposed scheme is robust, accurate, and stable for large density ratios, addressing key challenges 
in the discretization of NSCH models.
\end{abstract}

\end{frontmatter}

\section{Introduction}\label{sec:intro}

In recent decades, diffuse-interface, or phase-field, models have emerged as a powerful framework for simulating a wide range of multiphase flows, including contact line dynamics \cite{jacqmin2000contact,yue2011can}, complex fluids \cite{anderson1998diffuse,yue2004diffuse}, droplet dynamics \cite{espath2016energy,ten2024divergence}, fracture mechanics \cite{ambati2015review,wu2020phase}, and tumor growth \cite{oden2010general,garcke2018optimal}. For a comprehensive overview of diffuse-interface models in fluid mechanics, we refer to \cite{bray1994theory,anderson1998diffuse,cates2018theories}.
Within the domain of incompressible, isothermal, viscous two-phase flows, the prototypical phase-field model is the Navier-Stokes Cahn-Hilliard (NSCH) system. This model has proven to be a robust and versatile tool for simulating free-surface flows involving topological changes, surface tension effects, and significant density variations between phases. Typical applications include coalescence, and breakup of bubbles and droplets.

The development of Navier-Stokes Cahn-Hilliard (NSCH) models began with early efforts to address systems with matching densities, as introduced by Hohenberg and Halperin in 1977 \cite{hohenberg1977theory}. A derivation based on continuum mechanics was later provided by Gurtin et al. in 1996 \cite{gurtinmodel}. For an analysis of the asymptotic behavior of this model, we refer to \cite{gal2010asymptotic}, while the existence of weak solutions has been studied in \cite{colli2012global}. In recent years, the development of NSCH models has been extended to cases involving non-matching densities, with significant contributions from Lowengrub and Truskinovsky \cite{lowengrub1998quasi}, Boyer \cite{boyer2002theoretical}, Ding et al. \cite{ding2007diffuse}, Abels et al. \cite{abels2012thermodynamically}, Shen et al. \cite{shen2013mass}, Aki et al. \cite{aki2014quasi}, among others. Some of these NSCH models are established using elements of same underlying continuum mixture theory framework \cite{truesdell1960classical,truesdell1984historical}. Even though some of these models aim to describe the same physics, and are connected to the same underlying continuum mixture theory framework, the models are (seemingly) different. In fact, the models are often classified into so-called mass-averaged velocity models \cite{lowengrub1998quasi,shen2013mass,aki2014quasi} and volume-averaged velocity models \cite{boyer2002theoretical,ding2007diffuse,abels2012thermodynamically}. The existence for various models and classes of models underscores that consensus in the literature on the modeling of the non-matching density case has been missing. However, several of these models have been rigorously linked to sharp-interface formulations through sharp-interface limits \cite{abels2012thermodynamically,magaletti2013sharp,aki2014quasi}. In addition, two-phase NSCH models have been expanded in several aspects, e.g. to diffuse-interface models with $N$-phases \cite{boyer2014hierarchy,dong2018multiphase,ten2024thermodynamically,eikelder2024unified}, chemotaxis \cite{lam2018thermodynamically}, non-isothermal fluids \cite{zhao2023strong,brunk2024nonisothermal}, and dynamic boundary conditions \cite{giorgini2023two}.

In recent work of the second author \cite{eikelder2023unified,eikelder2024unified}, a unified framework for NSCH models with non-matching densities is proposed as a resolution to the longstanding inconsistencies among existing NSCH models. Although variations arise from constitutive choices, the unified framework leads to \textit{a single consistent NSCH mixture model}, invariant to the choice of fundamental variables. For example, the NSCH mixture model may be formulated in terms of mass-averaged and volume-averaged velocities, and these are shown to be equivalent through simple variable transformations. This framework is based on continuum mixture theory, as proposed by Truesdell and Toupin \cite{truesdell1960classical}. The framework shows that most existing NSCH models are only partially aligned with mixture theory, that materializes in inconsistencies in the balance laws or incompatibility in the single-fluid limit. By applying small modifications, these models can be shown to align with the framework in \cite{eikelder2023unified}, providing consistency, a natural connection to mixture theory, and a reduction to the incompressible Navier-Stokes equations in the single-fluid regime.

Despite significant progress, the discretization of NSCH models remains a major challenge. For the case of matching densities, extensive analysis has been conducted on energy-stable schemes \cite{feng2006fully,kay2007efficient,chen2016efficient,diegel2017convergence,brunk2023second}. However, for non-matching densities, the choice of the underlying NSCH formulation plays a crucial role in numerical algorithm development. For example, there are particular differences between existing volume-averaged and mass-averaged velocity formulations in the literature that are important for the development of numerical algorithms. First, in absence of mass transfer between fluids, the volume-averaged velocity formulations have a divergence-free velocity, whereas the mass-averaged velocity formulations in general do not. Although these models are (at the core) equivalent on the continuum level \cite{eikelder2023unified,eikelder2024unified}, the discretization of the formulation with a non-divergence-free velocity is often considered more challenging. Second, models based on volume-averaged velocities \cite{abels2012thermodynamically,ding2007diffuse,khanwale2023projection} exhibit weaker coupling than those based on mass-averaged velocities \cite{lowengrub1998quasi,shen2013mass}. Stronger coupling is sometimes suggested as a contributing factor to the increased difficulty of discretizing mass-averaged formulations \cite{abels2012thermodynamically}. As a result, most proposed numerical algorithms discretize NSCH models with volume-averaged velocities \cite{guillen2014splitting,garcke2016stable,chen2016efficient}. Energy-stable schemes within this class are typically discretizations of the model proposed by Abels et al. \cite{abels2012thermodynamically}.

Another crucial aspect of NSCH model formulations is the choice of the order parameter, with two common options being volume-fraction-based and concentration-based (also referred to as mass-fraction-based) order parameters. While these formulations are theoretically equivalent at the continuum level \cite{eikelder2023unified,eikelder2024unified}, discretizing NSCH models using concentration-based order parameters introduces additional complexities. Two primary challenges arise in this case: (i) the nonaffine mapping of the density in terms of the concentration, (ii) the occurrence of the density in both the Korteweg stress and the chemical potential. As a result, discretizations of concentration-based formulations are often more involved and require additional considerations \cite{guo2014numerical,guo2017mass}.


Beyond the choice of order parameters, handling large density ratios presents another major challenge in the numerical discretization of NSCH models. Developing an energy-stable numerical scheme for such cases is particularly difficult, especially when incorporating positive extensions of the density. Many existing methods prove energy stability but rely on positivity of the density. The actual implementation typically rely on phase field modifications that are not considered their theoretical proofs, cf. \cite{Gong18}. A key difficulty lies in ensuring that the numerical scheme remains stable with respect to a well-defined energy functional, even when phase field modifications like positive extensions of the density are employed.

In this work, we construct a provably energy-stable scheme where stability holds with respect to an energy functional based on positive extensions of the density. We develop a simple, fully-discrete, monolithic, mass-conservative, energy-stable method for the two-phase NSCH mixture model with non-matching densities. To establish the energy-stable method, we introduce an alternative -- equivalent --formulation that enables a fully-discrete, structure-preserving discretization. This formulation adopts mass-averaged velocity and volume fraction-based order parameters as unknowns, where the mass-averaged velocity is particularly advantageous for implementation due to the complexity of the NSCH mixture model’s momentum equation \cite{eikelder2023unified}.



The remainder of the paper is outlined as follows. In \cref{sec:model} we present the consistent NSCH model and analyze its properties. Additionally, we present the alternative but equivalent formulation that forms the basis for the discretization scheme. In \cref{{sec:scheme}} we introduce the fully-discrete numerical scheme and discuss its properties. Next, in \cref{sec:numerics} provide numerical examples. Finally, we close the paper with a summary and outlook in \cref{sec:conclusion and outlook}.

\section{The Navier-Stokes Cahn-Hilliard model}\label{sec:model}

This section is concerned with the Navier-Stokes Cahn-Hilliard model. \cref{subsec:gov eq} presents the governing equations, and \cref{subsec:reform} provides an equivalent alternative formulation. Finally, \ref{subsec:dim less} presents the non-dimensional formulation.

\subsection{Governing equations and properties}\label{subsec:gov eq}

We consider the Navier-Stokes Cahn-Hilliard system with non-matching densities given by:
\begin{subequations}\label{eq:sys1}
  \begin{align}
 \dt\rho &+ \div(\rho\vv) = 0, \label{eq:sys1: mass}\\
 \dt(\rho\vv) &+ \div(\rho\vv\otimes\vv) - \div \mathbf{S} + \nabla p + \phi\nabla\mu = \rho\mathbf{g},\label{eq:sys1: mom}\\
 \dt\phi &+ \div(\phi\vv) - \div(\mathbf{M}\nabla(\mu+\alpha p)) = 0, \label{eq:sys1: phase}\\
 \mu &+ \gamma\Delta\phi - f'(\phi) = 0, \label{eq:sys1: chem}
\end{align}
\end{subequations}
in domain $\Omega \subset \mathbb{R}^d$ with dimension $d=2,3$, boundary $\Gamma$ and unit outward normal $\mathbf{n}$. The unknowns are the velocity $\vv: \Omega \rightarrow \mathbf{R}^d$, phase field variable $\phi: \Omega \rightarrow \mathbf{R}$ and pressure $p: \Omega \rightarrow \mathbf{R}$ subject to the initial conditions $\vv(\mathbf{x},0) = \vv_0(\mathbf{x})$ and $\phi(\mathbf{x},0) = \phi_0(\mathbf{x})$. The phase field variable $\phi$ equals the volume fraction difference, i.e. $\phi=1$ represents the first constituent and $\phi=-1$ the second. The density and viscosity are parameterized via 
\begin{subequations}
  \begin{align}
  \rho \equiv\rho(\phi)=&~\rho_1\frac{1+\phi}{2}+\rho_2\frac{1-\phi}{2}, \label{eq: def rho}\\
  \eta \equiv\eta(\phi)=&~\eta_1\frac{1+\phi}{2}+\eta_2\frac{1-\phi}{2}  \label{eq: def eta}  
\end{align}
\end{subequations}
for given positive ($\geq 0$) constant constituent densities $\rho_1$ and $\rho_2$ and viscosities $\eta_1$ and $\eta_2$. The Cauchy stress is given by $\mathbf{S} = \eta (2 \nabla^s \vv + \lambda (\div \vv)\mathbf{I})$ where $\nabla^s \vv = (\nabla \vv + (\nabla \vv)^T)/2$ is the symmetric velocity gradient, and $\lambda = -2/d$. Furthermore, the gravitational force vector is $\mathbf{g} = -g \mathbf{j} = \nabla y$ with $g$ the gravitational constant and $\mathbf{j}$ the vertical unit vector. We introduce the constants $\arho = (\rho_1+\rho_2)/2$ and $\jrho=(\rho_1-\rho_2)/2$, $\alpha = -\jrho/\arho = (\rho_2-\rho_1)/(\rho_1+\rho_2)$. The quantity $\mu$ denotes the chemical potential, and $\mathbf{M}$ is the mobility tensor. The mobility tensor is symmetric, positive semi-definite, and is degenerate, i.e. it vanishes in the single fluid regime: $\mathbf{M} = 0\mathbf{I}$ when $\phi = \pm 1$. Equations \eqref{eq:sys1: mass} and \eqref{eq:sys1: mom} describe the mass and momentum balance of the mixture.

\begin{remark}[Mixture velocity]
    From the perspective of continuum mixture theory, the velocity $\vv$ in the system \eqref{eq:sys1} is the mass-averaged velocity. It is well-known that this velocity is not a divergence-free velocity in general, see e.g. \cite{eikelder2023unified}. 
\end{remark}

\begin{remark}[Mobility]
  In this paper we restrict to $\mathbf{M}=\mathbf{M}(\phi)$, however the analysis directly carries over to more general dependencies such as $\mathbf{M}=\mathbf{M}(\phi,\nabla \phi)$.
\end{remark}

The system has an additionally underlying structure of balance laws and dissipation identities. Namely under suitable boundary conditions the system conserves mass and dissipates energy via
\begin{align}
   \ddt \int_\Omega \phi(t) = 0,\qquad \ddt \mathcal{E}(\phi,\vv) = -\mathcal{D}_\phi(\mu+\alpha p,\vv ),
\end{align}
with the energy $\mathcal{E}(\phi,\vv)$ and dissipation rate $\mathcal{D}_\phi(\mu+\alpha p,\vv )$ given by
\begin{subequations}
\begin{align}
 \mathcal{E}(\phi, \vv):=&~\int_\Omega \left( \frac{\gamma}{2}\snorm{\nabla\phi}^2 + f(\phi)\right) + \frac{\rho(\phi)}{2}\snorm{\vv}^2 + \rho(\phi)y = \int_\Omega \Psi(\phi,\nabla\phi) + K(\phi,\vv) + G(\phi)  , \label{eq:defEnergy}\\
 \mathcal{D}_\phi(\mu+\alpha p,\vv ):=&~ \int_\Omega \nabla(\mu+\alpha p)\cdot\mathbf{M}\nabla(\mu+\alpha p) + \mathbf{S}:\nabla\vv,  \label{eq:defDissipation}
\end{align}
\end{subequations}
where we decomposed the energy density of $\mathcal{E}$ into the kinetic energy density $K$, the energy density due to gravity $G$ and the free energy density $\Psi$. Note that $y$ is the vertical coordinate, i.e. $\nabla y = \mathbf{j}.$


\begin{remark}[Conservative form]
  In absence of gravity ($\mathbf{g}=0$), the model may be written in conservative form by the inserting the Korteweg tensor identity:
  \begin{align}
    \phi \nabla \mu =&~ \div \left( \gamma \nabla \phi \otimes \nabla \phi - \left(\mu \phi - f(\phi) - \frac{\gamma}{2}\snorm{\nabla\phi}^2\right)\mathbf{I}\right).
  \end{align}
\end{remark}


\subsection{Equivalent alternative formulation}\label{subsec:reform}

The evolution of the energy of the system \eqref{eq:sys1} results from a linear combination of the equations \eqref{eq:sys1: mass}-\eqref{eq:sys1: chem} with the weights $- |\vv|^2/2 + 2p/(\rho_1+\rho_2)$, $\vv$, $\mu+\alpha p$ and $-\partial_t \phi$, respectively. Using standard velocity-pressure function spaces, the first weight is not an element of such a function space. To circumvent this issue, we introduce an alternative -- but equivalent -- formulation for which the energy evolution follows from standard weighting function spaces.

First, we replace the mass balance law \eqref{eq:sys1: mass} by the balance law of $\div \vv$. The expression of  $\div \vv$ follows from a linear combination of the mass balance law \eqref{eq:sys1: mass} and the phase-field evolution equation \eqref{eq:sys1: phase}:
\begin{align}\label{eq:ID: mass}
    0=&~\frac{2}{\rho_1+\rho_2}\left(\dt\rho+ \div(\rho\vv) \right)
    +\alpha\left(\dt\phi + \div(\phi\vv) - \div(\mathbf{M}\nabla(\mu+\alpha p))\right)\nn\\
    =&~ \div \vv - \alpha \div \left(\mathbf{M}\nabla(\mu+\alpha p)\right),
\end{align}
where we have used the identities:
\begin{subequations}
    \begin{align}
        \frac{2}{\rho_1+\rho_2}\dt\rho + \alpha \dt\phi  =&~0, \\
        \frac{2}{\rho_1+\rho_2}\div(\rho\vv) + \alpha \div(\phi\vv)  =&~0,      
    \end{align}
\end{subequations}
which follow from \eqref{eq: def rho}.

Second, we use the mass balance \eqref{eq:sys1: mass} to rewrite the momentum balance law \eqref{eq:sys1: mom}:
\begin{align}\label{eq:ID: mom}
    0=&~\dt(\rho\vv) + \div(\rho\vv\otimes\vv) - \div \mathbf{S} + \nabla p + \phi\nabla\mu - \rho\mathbf{g}
    - \frac{1}{2}\vv\left(\dt\rho + \div(\rho\vv)\right)\nn\\
    =&~ \frac{\vv}{2}\dt\rho + \rho\dt\vv + \frac{1}{2} \vv \div(\rho \vv) + \rho \vv\cdot \nabla \vv - \div \mathbf{S} + \nabla p + \phi\nabla\mu - \rho\mathbf{g},
\end{align}
where we have used the identities:
\begin{subequations}
    \begin{align}
        \dt(\rho\vv) - \frac{1}{2}\vv\dt\rho =&~\frac{\vv}{2}\dt\rho + \rho\dt\vv, \\
        \div(\rho\vv\otimes\vv) - \frac{1}{2}\vv \div(\rho\vv) =&~\frac{1}{2} \vv \div(\rho \vv) + \rho \vv\cdot \nabla \vv.      
    \end{align}
\end{subequations}

Utilizing \eqref{eq:ID: mass} and \eqref{eq:ID: mom} we recast the system \eqref{eq:sys1} into the equivalent strong form:
\begin{subequations}\label{eq:sys2}
  \begin{align}
 \div \vv  &- \alpha \div \left(\mathbf{M}\nabla(\mu+\alpha p)\right) = 0, \label{eq:sys2: div}\\
 \frac{\vv}{2}\dt\rho &+ \rho\dt\vv + \frac{1}{2} \vv \div(\rho \vv) + \rho \vv\cdot \nabla \vv - \div \mathbf{S} + \nabla p + \phi\nabla\mu = \rho\mathbf{g},\label{eq:sys2: mom}\\
 \dt\phi &+ \div(\phi\vv) - \div(\mathbf{M}\nabla(\mu+\alpha p)) = 0, \label{eq:sys2: phase}\\
 \mu &+ \gamma\Delta\phi - f'(\phi) = 0.\label{eq:sys2: chem}
\end{align}
\end{subequations}
The energy evolution of this formulation now follows from using standard weights.
\begin{lemma}[Energy evolution]
The energy evolution of \eqref{eq:sys2} follows from a linear combination of \eqref{eq:sys2: div}-\eqref{eq:sys2: chem} with the weights: $p+g y\arho$, $\vv$, $\mu+ gy\jrho$, $-\partial_t \phi$.
\end{lemma}
\begin{proof}
  Multiplying \eqref{eq:sys2: div} with $p+g y\arho$ provides:
  \begin{align}\label{eq:proof1: div}
 0 = (p+g y\arho)\div \vv  - \alpha (p+g y\arho) \div \left(\mathbf{M}\nabla(\mu+\alpha p)\right).
 \end{align}
 Next, taking the inner product of the momentum equation \eqref{eq:sys2: mom} with $\vv$ yields:
  \begin{align}\label{eq:proof1: mom}
    0 &= \partial_t K + \div (K \vv) - \vv \cdot \div \mathbf{S} + \vv\cdot \nabla p + \phi\vv\cdot \nabla\mu - \rho\vv\cdot\mathbf{g},
  \end{align}
  where we have used the identities:
  \begin{subequations}
    \begin{align}
      \partial_t K  =&~ \vv\cdot \left( \frac{\vv}{2}\dt\rho + \rho\dt\vv \right)\\      
      \div (K \vv) =&~ \vv\cdot \left(\frac{1}{2} \vv \div(\rho \vv) + \rho \vv\cdot \nabla \vv \right)
    \end{align}
  \end{subequations}
  Finally, multiplying \eqref{eq:sys2: phase} with $\mu+gy\jrho$ and $-\partial_t \phi$, and subsequently adding the results provides:
  \begin{align}\label{eq:proof1: phase+chem}
      0 = &~\partial_t G + (\mu+gy \jrho) \div(\phi\vv) - (\mu+ gy \jrho) \div(\mathbf{M}\nabla(\mu+\alpha p)) \nn\\
  &~+ \partial_t \Psi - \div \left( \gamma \nabla \phi \cdot \dot{\phi} \right)+ \div \left( (\gamma \nabla \phi \otimes \nabla \phi)\vv \right),
  \end{align}
  where we have used the identities:
  \begin{subequations}
      \begin{align}
      \partial_t G = &~ gy \jrho \partial_t \phi \\
      \partial_t \Psi - \div \left( \gamma \nabla \phi \cdot \dot{\phi} \right)+ \div \left( (\gamma \nabla \phi \otimes \nabla \phi)\vv \right) =&~ - \partial_t \phi \gamma \Delta \phi + \partial_t \phi f'(\phi).
  \end{align}
  \end{subequations}
Adding \eqref{eq:proof1: div}, \eqref{eq:proof1: mom} and \eqref{eq:proof1: phase+chem} gives:
 \begin{align}\label{eq:proof1: total}
 0 =&~ \partial_t (K + \Psi + G) + \div ((K+\Psi + G) \vv) - \div (\mathbf{S}\vv) + \div (p \vv)   - \div \left((\mu+\alpha p) \mathbf{M}\nabla(\mu+\alpha p)\right)   \nn\\
 &~  + \div\left(\mathbf{K}\vv\right) - \div \left( \gamma \nabla \phi \cdot \dot{\phi} \right)   +\mathbf{S}:\nabla \vv + \nabla(\mu+\alpha p) \cdot\left(\mathbf{M}\nabla(\mu+\alpha p)\right),
 \end{align}
 with $\mathbf{K}$ the Korteweg tensor given by:
 \begin{align}
     \mathbf{K} = (\mu\phi-\Psi)\mathbf{I}+\gamma \nabla \phi \otimes \nabla \phi,
 \end{align}
 and the $\dot{\phi} = \partial_t \phi + \vv \cdot \nabla \phi$ denoting the convective derivative.
 Integration over $\Omega$ provides:
 \begin{align}
   \dfrac{{\rm d}}{{\rm d}t}\mathcal{E}(\phi,\vv) = -\mathcal{D}_\phi(\mu+\alpha p,\vv ) + \mathcal{B},
 \end{align}
 where $\mathcal{B}$ is the boundary contribution:
 \begin{align}
   \mathcal{B} = - \int_{\partial\Omega} (K+\Psi) \vv\cdot \mathbf{n} - (\mathbf{S}-p \mathbf{I}) \vv\cdot \mathbf{n} + (\mu + \alpha p) \mathbf{M}\nabla (\mu + \alpha p)\cdot \mathbf{n} + \mathbf{K} \vv\cdot\mathbf{n}  -\gamma \dot{\phi} \nabla \phi \cdot \mathbf{n}. 
 \end{align}
\end{proof}

In the following we will only consider the following sets of boundary conditions
\begin{itemize}
    \item $\Omega$ is a hypercube and identified with the d-dimensional torus, i.e. we impose periodic boundary conditions.
    \item For the phase-field $\nabla\phi\cdot\mathbf{n}\vert_{\partial\Omega}=\mathbf{M}\nabla(\mu+\alpha p) \cdot\mathbf{n}\vert_{\partial\Omega}=0.$ For the velocity we decompose $\partial\Omega=\partial\Omega_1\cup \partial\Omega_2$ such that $\vv\vert_{\partial\Omega_1}=\mathbf{0}$ and  $\vv\cdot\mathbf{n}\vert_{\partial\Omega_2}=\mathbf{0}.$
\end{itemize}

To derive the weak formulation we introduce the notation $\la a,b\ra := \int_\Omega ab$, for arbitrary functions $a,b:\Omega \rightarrow \mathbb{R}^k$, $k\in\mathbb{N}_+$. Guided by the alternative strong form of the momentum equation \eqref{eq:ID: mom}, we utilize the skew-symmetric form of the convection term in the weak formulation of the momentum equation:
\begin{align}
   \mathbf{c}_{skw}(\u,\vv,\w) := &~ \frac{1}{2}\la (\u\cdot\nabla)\vv,\w \ra  - \frac{1}{2}\la (\u\cdot\nabla)\w,\vv \ra,
\end{align}
where we note the identity:
\begin{align}
    \la \w ,\frac{1}{2} \vv \div(\rho \vv) + \rho \vv\cdot \nabla \vv \ra = \mathbf{c}_{skw}(\rho\vv,\vv,\w).
\end{align}

With this we can recast the system into a variational formulation.

\begin{lemma}[Energy-stable variation formulation]\label{lem:variational_TC}
Every smooth solution satisfies the variational formulation
\begin{align}
  \la \dt\phi,\psi \ra &- \la \phi\vv, \nabla\psi\ra + \la \mathbf{M}(\phi)\nabla(\mu+\alpha p),\nabla\psi \ra = 0, \\
  \la \mu,\xi \ra &- \gamma\la \nabla\phi,\nabla\xi \ra - \la f'(\phi),\xi \ra = 0,\\
  \la \frac{\vv}{2}\dt\rho &+ \rho\dt\vv,\w \ra +  \mathbf{c}_{skw}(\rho\vv,\vv,\w) + \la \mathbf{S}(\phi,\nabla\vv),\nabla\w \ra \\
  &- \la p,\div(\w)\ra + \la \phi\nabla\mu,\w \ra + \la \rho(\phi)\mathbf{j},\w \ra = 0, \\
  \la \div(\vv),q \ra &+ \alpha\la \mathbf{M}(\phi)\nabla(\mu+\alpha p),\nabla q \ra  = 0
\end{align}
for smooth test functions $(\psi,\xi,\w,q)$ with mean-free $q$. Furthermore, conservation of mass and total density and energy dissipation holds
    \begin{align}
          \ddt \la \phi(t), 1\ra = 0,\qquad \ddt \la \rho(\phi(t)), 1\ra = 0,\qquad \ddt \mathcal{E}(\phi(t),\vv(t)) = - \mathcal{D}_{\phi(t)}(\mu(t)+\alpha p(t),\vv(t)). \label{eq:massdissipation_cont}
    \end{align}
\end{lemma}
\begin{proof}
    Conservation of mass follows from $\psi=1$ using that $\nabla 1=0$.  Next we focus on the energy dissipation relation and denote the mean-value of $y$ by $\bar y:=\la y,1 \ra$. Taking $q=p+g (y-\bar y)\arho$ provides:
  \begin{subequations}\label{eq:proof2: div}
  \begin{align}
 0=&~\la \div(\vv),p \ra + \la \mathbf{M}(\phi)\nabla(\mu+\alpha p),\nabla (\alpha p) \ra \nn\\
 &~+ g \arho \la \div(\vv),y-\bar y \ra -  g \jrho \la \mathbf{M}(\phi)\nabla(\mu+\alpha p),\mathbf{j} \ra \\
 =&~\la \div(\vv),p \ra + \la \mathbf{M}(\phi)\nabla(\mu+\alpha p),\nabla (\alpha p) \ra \nn\\
 &~+ g \arho \la \div(\vv),y\ra -  g \jrho \la \mathbf{M}(\phi)\nabla(\mu+\alpha p),\mathbf{j} \ra,
 \end{align}
 \end{subequations}
 where we used $\la \div(\vv),\bar y \ra=0$ and $\nabla \bar y =0.$
    Taking $\w=\vv$ provides:  
    \begin{align}\label{eq:proof2: mom}
    0 =&~ \dfrac{{\rm d}}{{\rm d}t} \la K (\phi,\vv),1\ra + \la \mathbf{S}(\phi,\nabla\vv),\nabla\vv \ra - \la p,\div(\vv)\ra + \la \phi\nabla\mu,\vv \ra + \la \rho(\phi)\mathbf{j}, \vv\ra,
  \end{align}
  where we have used the identities:
  \begin{subequations}
    \begin{align}
    \la \tfrac{1}{2}\vv\partial_t\rho + \rho\partial_t\vv,\vv \ra  =&~\dfrac{{\rm d}}{{\rm d}t} \la K (\phi,\vv),1\ra \\
    \mathbf{c}_{skw}(\rho\vv,\vv,\vv) = &~0.
    \end{align}
  \end{subequations}
      Finally, taking $\psi = \mu+gy\jrho $ and $\xi=-\partial_t\phi $, and subsequently adding the results provides:
\begin{subequations}\label{eq:proof2: phase}
    \begin{align}
     0=&~ \dfrac{{\rm d}}{{\rm d}t}\la\Psi(\phi),1\ra - \la \phi\vv, \nabla\mu \ra+ \la \mathbf{M}(\phi)\nabla(\mu+\alpha p),\nabla\mu  \ra,\nn\\
      &+\dfrac{{\rm d}}{{\rm d}t} \la G(\phi),1\ra -g\jrho \la \phi\vv, \mathbf{j}\ra + g \jrho\la \mathbf{M}(\phi)\nabla(\mu+\alpha p),\mathbf{j} \ra,
    \end{align}
\end{subequations}
  where we have used the identities:
  \begin{subequations}
      \begin{align}
         \gamma\la \nabla\phi,\nabla \partial_t\phi \ra + \la \partial_t f(\phi),1 \ra  =&~ \dfrac{{\rm d}}{{\rm d}t} \la\Psi(\phi),1\ra ,\\
       \la \partial_t \phi,gy\jrho  \ra = &~ \dfrac{{\rm d}}{{\rm d}t} \la G(\phi),1\ra,
  \end{align}
  \end{subequations}
 Addition of \eqref{eq:proof2: div}, \eqref{eq:proof2: mom} and \eqref{eq:proof2: phase} provides:
  \begin{align}
      \dfrac{{\rm d}}{{\rm d}t} \mathcal{E}(\phi,\vv) =&~  - \la \mathbf{S}(\phi,\nabla\vv),\nabla\vv \ra - \la \mathbf{M}(\phi)\nabla(\mu+\alpha p),\nabla(\mu+\alpha p)  \ra \nn\\
 = &~ - \mathcal{D}_{\phi}(\mu+\alpha p,\vv),
  \end{align}
  where we have used the identity:
\begin{align}
    -g\jrho \la \phi\vv, \mathbf{j}\ra+ g \arho \la \div(\vv),y \ra   + \la \rho(\phi)g\mathbf{j}, \vv\ra = 0,
\end{align}
  where the last identity follows from $\vv\cdot\mathbf{n}\vert_{\partial\Omega}=0$.
\end{proof}

\section{Numerical scheme \& structural properties}\label{sec:scheme}
In this section, we propose a fully discrete finite element method and prove our main result on the structure-preservation properties. First, in \cref{subsec: time discr} we present the time integration, and subsequently, in \cref{subsec: space discr}, we provide the spatial discretization to establish the fully-discrete methodology.

\subsection{Time discretization}\label{subsec: time discr}
Let us introduce the relevant notation and assumptions for our time discretization strategy. We partition the time interval $[0,T]$ uniformly with a time parameter $\tau>0$ and introduce $\Itau:=\{0=t^0,t^1=\tau,\ldots, t^{n_T}=T\}$, where $n_T=\tfrac{T}{\tau}$ is the absolute number of time steps. We denote by $\Pi^1_c(\Itau), \Pi^0(\Itau)$ the spaces of continuous piece-wise linear and piecewise constant functions on $\Itau$. We introduce $c^{n+1}:=c(t^{n+1})$ and $c^n:=c(t^n)$ . We introduce $c^*:=c(t^*)$ as a placeholder, which allows for every reasonable approximation in time. Typical options are $c^*\in\{c^n,c^{n+1}\}.$ Finally, we introduce the time difference and the discrete time derivative via
\begin{equation*}
	d^{n+1}c = c^{n+1} - c^n, \qquad d^{n+1}_\tau c = \frac{c^{n+1}-c^n}{\tau}.
\end{equation*}
To treat the convex and concave nature of the potential $f(\phi)$ we use time-averages in the spirit of \cite{Brunk2023}, i.e.
\begin{align}
 f'(\phi^{n+1},\phi^n) := \frac{1}{\tau}\int_{t^{n}}^{t^{n+1}} f'(\phi(s)) ds, \qquad \phi(s) = \frac{\phi^{n+1}-\phi^n}{\tau}(s-t^{n+1}) + \phi^{n+1}.  \label{eq:timeavg} 
\end{align}
We define extensions $\widetilde \rho(\phi), \widetilde\eta(\phi)$ by extending $\rho(\phi),\eta(\phi)$ by $\rho_1,\eta_1$ for $\phi < -1$ and $\rho_2,\eta_2$ for $\phi > 1$. We emphases that also other positive extensions can be used.

First we propose a time-integration scheme:
\begin{subequations}\label{eq:sys time}
  \begin{align}
 \div (\vv^{n+1})  &- \alpha \div \left(\mathbf{M}(\phi^*)\nabla(\mu^{n+1}+\alpha p^{n+1})\right) = 0, \label{eq:sys time: div}\\
 \frac{\vv^{n+1}}{2}d^{n+1}_\tau \widetilde{\rho} &+ \widetilde{\rho} d^{n+1}_\tau\vv + \frac{1}{2} \vv^{n+1} \div(\rho^* \vv^*) + \rho^* \vv^*\cdot \nabla \vv^{n+1}\nn\\
 &- \div \mathbf{S}(\phi^*,\nabla \vv^{n+1}) + \nabla p^{n+1} + \phi^*\nabla\mu^{n+1} = \rho(\phi^*)\mathbf{g},\label{eq:sys time: mom}\\
 d^{n+1}_\tau\phi &+ \div(\phi^*\vv^{n+1}) - \div(\mathbf{M}(\phi^*)\nabla(\mu^{n+1}+\alpha p^{n+1})) = 0, \label{eq:sys time: phase}\\
 \mu^{n+1} &+ \gamma\Delta\phi^{n+1} - f'(\phi^{n+1},\phi^n) = 0.\label{eq:sys time: chem}
\end{align}
\end{subequations}

Note that one can show that solutions of \eqref{eq:sys time} conserve of mass, total densities and satisfy the energy dissipation. The proof follows the same lines for Lemma \ref{lem:variational_TC} and is postponed to the proof for the fully discrete scheme, i.e. Theorem \ref{thm:discreten_stable} which we will consider below.

\subsection{Fully-discrete methodology}\label{subsec: space discr}

For the spatial discretisation we require that $\Th$ is a geometrically conforming partition of $\Omega$ into simplices where $h$ is the maximal diameter of the triangles in $\Th$. 
%
%
The space of continuous and piecewise linear and quadratic functions over $\Th$ as well as the mean free is introduced via
\begin{subequations}\label{eq:defFespace}
\begin{align}
	\Vh &:= \{v \in H^1(\Omega)\cap C^0(\bar\Omega) : v|_K \in P_1(K) \quad \forall K \in \Th\},\\
	\Xh &:= \{\vv \in H^1(\Omega)^d\cap C^0(\bar\Omega)^d : \vv|_K \in P_2^d(K) \quad \forall K \in \Th\}\cap\{\vv\vert_{\partial\Omega_1}=0, \vv\cdot \mathbf{n}\vert_{\partial\Omega_2}=0  \},\\
	\Qh &:= \{v \in \Vh : \la v, 1\ra=0\}.
\end{align}
\end{subequations}
Here $P_k(K)$ denote the space of polynomials with maximal degree $k$ on $K$.
In the case of periodic boundary conditions the boundary incorporated in $\Xh$ are neglected. 
 
This variational formulation allows to directly deduce a structure-preserving approximation.

\begin{problem}[Fully-discrete method]\label{prob:scheme}
    Let $(\phi_{h}^0,\vv_{h}^0)\in \Vh\times\Xh$ be given. We seek function $(\phi_h,\vv_h)\in \Pi_c^1(\Itau;\Vh\times\Xh)$ and $(\mu_h,p_h)\in \Pi^0(\Itau;\Vh\times\Qh)$ such that
    \begin{align}
  \la \dtau\phi_h,\psi_h \ra &- \la \phi_h^*\vv_h^{n+1}, \nabla\psi_h\ra + \la \mathbf{M}(\phi_h^*)\nabla(\mu_h^{n+1}+\alpha p_h^{n+1}),\nabla\psi_h \ra = 0, \label{eq:scheme1}\\
  \la \mu_h^{n+1},\xi_h \ra &- \gamma\la \nabla\phi_h^{n+1},\nabla\xi_h \ra - \la f'(\phi_h^{n+1},\phi_h^n),\xi_h \ra = 0,\label{eq:scheme2}\\
  \la \tfrac{\vv_h^{n+1}}{2}\dtau\widetilde\rho_h &+ \widetilde\rho_h^n\dtau\vv_h,\w_h \ra +  \mathbf{c}_{skw}(\rho_h^*\vv_h^*,\vv_h^{n+1},\w_h^{n+1})\nn \\
&+ \la \mathbf{S}(\phi^*,\nabla\vv_h^{n+1}),\nabla\w_h \ra - \la p_h^{n+1},\div(\w_h)\ra + \la \phi_h^*\nabla\mu_h^{n+1},\w_h \ra + \la \rho(\phi^*_h)g\mathbf{j}, \w_h\ra = 0, \label{eq:scheme3}\\
  \la \div(\vv_h^{n+1}),q_h \ra &+ \alpha\la \mathbf{M}(\phi_h^*)\nabla(\mu_h^{n+1}+\alpha p_h^{n+1}),\nabla q_h \ra= 0 \label{eq:scheme4}
\end{align}
holds for all $(\psi_h,\xi_h,\w_h,q_h)\in\Vh\times\Vh\times\Xh\times\Qh$ and for all $0\leq n < n_T.$
\end{problem}
The structure-preserving properties now follow in a manner similar to the continuous setting. In the discrete energy, the kinetic energy component is formulated using a clipped phase-field variable.
\begin{theorem}[Structure-preserving properties]\label{thm:discreten_stable}
    Every solution of Problem \ref{prob:scheme} satisfies the conservation of mass, total density and the energy dissipation law, i.e. it holds for all $0 \leq n\leq n_T-1$ that
    \begin{subequations}
    \begin{align}
      \la \phi_h^{n+1},1 \ra &= \la \phi_{0,h},1 \ra, \qquad \la \rho(\phi_h^{n+1}),1 \ra = \la \rho(\phi_{0,h}),1 \ra, \label{eq:conservation_mass_discrete}\\
       \widetilde{\mathcal{E}}(\phi_h^{n+1},\vv_h^{n+1}) &+ \tau\mathcal{D}_{\phi^*_h}(\mu_h^{n+1}+\alpha p_h^{n+1},\vv_h^{n+1}) \leq  \widetilde{\mathcal{E}}(\phi_h^{n},\vv_h^{n})\label{eq:dissipation_discrete}
    \end{align}
    \end{subequations}
    with discrete energy $\widetilde{\mathcal{E}}(\phi,\vv):=\int_\Omega \frac{\gamma}{2}\snorm{\nabla\phi}^2 + f(\phi) + \frac{\widetilde\rho}{2}\snorm{\vv}^2 + g\rho y.$
    
\end{theorem}
\begin{proof}
    Conservation of mass follows again by insertion of $\psi_h=1\in\Vh.$ The conservation of the density is a direct consequence from conservation of mass, since $\rho(\phi)$ is affine in $\phi$, i.e. $\rho'(\phi)=\jrho=\text{const}$. 
    
    Next we focus on the energy dissipation relation. We introduce the clipped kinetic energy density via $\widetilde{K}(\phi,\vv):=\frac{\widetilde\rho}{2}\snorm{\vv}^2.$  As before we denote with $\bar y = \la y, 1\ra.$  Taking $q_h=p_h^{n+1}+g (y-\bar y)\arho  \in \Qh$ provides:
\begin{subequations}\label{eq:proof3: div}
    \begin{align}
     0=&~\la \div(\vv_h^{n+1}),p_h^{n+1} \ra + \la \mathbf{M}(\phi_h^*)\nabla(\mu_h^{n+1}+\alpha p_h^{n+1}),\nabla (\alpha p_h^{n+1}) \ra \nn \\
     &~+ g \arho \la \div(\vv_h^{n+1}),y-\bar y \ra -  g \jrho \la \mathbf{M}(\phi_h^*)\nabla(\mu_h^{n+1}+\alpha p_h^{n+1}),\mathbf{j} \ra\\
     =&~\la \div(\vv_h^{n+1}),p_h^{n+1} \ra + \la \mathbf{M}(\phi_h^*)\nabla(\mu_h^{n+1}+\alpha p_h^{n+1}),\nabla (\alpha p_h^{n+1}) \ra \nn\\
     &~+ g \arho \la \div(\vv_h^{n+1}),y \ra -  g \jrho \la \mathbf{M}(\phi_h^*)\nabla(\mu_h^{n+1}+\alpha p_h^{n+1}),\mathbf{j} \ra,
     \end{align}
 \end{subequations}
 where we used that $\la \div(\vv_h^{n+1}),\bar y \ra=0$ and $\nabla\bar y =0.$
    Taking $\w_h=\vv_h^{n+1}\in\Xh$ provides:  
    \begin{align}\label{eq:proof3: mom}
    0 =&~ \frac{1}{\tau}\la \widetilde{K}(\phi_h^{n+1},\vv_h^{n+1}) - \widetilde{K}(\phi_h^{n},\vv_h^{n}),1\ra + \tau\frac{\widetilde\rho_h^n}{2}\norm{\dtau\vv_h}_0^2 \nn\\
    &~+ \la \mathbf{S}(\phi^*_h,\nabla\vv_h^{n+1}),\nabla\vv_h^{n+1} \ra - \la p_h^{n+1},\div(\vv_h^{n+1})\ra + \la \phi_h^*\nabla\mu_h^{n+1},\vv_h^{n+1} \ra + \la \rho(\phi^*_h)\mathbf{j}, \vv_h^{n+1}\ra,
  \end{align}
  where we have used the identities:
  \begin{subequations}
    \begin{align}
    \la \tfrac{1}{2}\vv_h^{n+1}\dtau\widetilde\rho_h + \widetilde\rho_h^n\dtau\vv_h^{n+1},\vv_h^{n+1} \ra  =&~\frac{1}{\tau}\la \widetilde{K}(\phi_h^{n+1},\vv_h^{n+1}) - \widetilde{K}(\phi_h^{n},\vv_h^{n}),1\ra + \tau\frac{\widetilde\rho_h^n}{2}\norm{\dtau\vv_h}_0^2\\
    \mathbf{c}_{skw}(\rho_h^*\vv_h^*,\vv_h^{n+1},\vv_h^{n+1}) = &~0.
    \end{align}
  \end{subequations}
      Finally, taking $\psi_h = \mu_h^{n+1}+gy\jrho \in \Vh$ and $\xi_h=-d_\tau^{n+1} \phi_h \in \Vh$, and subsequently adding the results provides:
\begin{subequations}\label{eq:proof3: phase}
    \begin{align}
     0=&~ \frac{1}{\tau}\la \Psi(\phi_h^{n+1}) - \Psi(\phi_h^{n}),1\ra+ \tau\frac{\gamma}{2}\norm{\nabla\dtau\phi_h}_0^2 - \la \phi_h^*\vv_h^{n+1}, \nabla\mu_h^{n+1} \ra+ \la \mathbf{M}(\phi_h^*)\nabla(\mu_h^{n+1}+\alpha p_h^{n+1}),\nabla\mu_h^{n+1}  \ra,\nn\\
      &+\frac{1}{\tau}\la G(\phi_h^{n+1}) - G(\phi_h^{n}),1\ra -g\jrho \la \phi_h^*\vv_h^{n+1}, \mathbf{j}\ra + g \jrho\la \mathbf{M}(\phi_h^*)\nabla(\mu_h^{n+1}+\alpha p_h^{n+1}),\mathbf{j} \ra,
    \end{align}
\end{subequations}
  where we have used the identities:
  \begin{subequations}
      \begin{align}
        \la \dtau \phi_h,f'(\phi_h^{n+1},\phi_h^n) \ra=&~ \frac{1}{\tau}\int_{t^n}^{t^{n+1}} \la f'(\phi_h),\dt\phi_h\ra ds = \frac{1}{\tau}\la f(\phi_h^{n+1})-f(\phi_h^n),1 \ra,\\
      \gamma\la \nabla\phi_h^{n+1},\nabla d_\tau^{n+1} \phi_h \ra + \frac{1}{\tau}\la f(\phi_h^{n+1})-f(\phi_h^n),1 \ra  =&~ \frac{1}{\tau}\la \Psi(\phi_h^{n+1}) - \Psi(\phi_h^{n}),1\ra+ \tau\frac{\gamma}{2}\norm{\nabla\dtau\phi_h}_0^2 ,\\
       \la \dtau\phi_h,gy\jrho  \ra = &~ \frac{1}{\tau}\la G(\phi_h^{n+1}) - G(\phi_h^{n}),1\ra,
  \end{align}
  \end{subequations}
  and that $\partial_t \phi_h$ is a piecewise constant in time. Addition of \eqref{eq:proof3: div}, \eqref{eq:proof3: mom} and \eqref{eq:proof3: phase} provides:
  \begin{align}
      \frac{1}{\tau}(\widetilde{\mathcal{E}}(\phi_h^{n+1},\vv_h^{n+1}) - \widetilde{\mathcal{E}}(\phi_h^{n},\vv_h^{n})) =&~  - \la \mathbf{S}(\phi^*_h,\nabla\vv_h^{n+1}),\nabla\vv_h^{n+1} \ra - \la \mathbf{M}(\phi_h^*)\nabla(\mu_h^{n+1}+\alpha p_h^{n+1}),\nabla(\mu_h^{n+1}+\alpha p_h^{n+1})  \ra \nn\\
 &~ - \tau\frac{\widetilde\rho_h^n}{2}\norm{\dtau\vv_h}_0^2 - \tau\frac{\gamma}{2}\norm{\nabla\dtau\phi_h}_0^2\nn\\
 \leq &~ - \mathcal{D}_{\phi^*_h}(\mu_h^{n+1}+\alpha p_h^{n+1},\vv_h^{n+1}),
  \end{align}
  where we have used the identity:
\begin{align}
    -g\jrho \la \phi_h^*\vv_h^{n+1}, \mathbf{j}\ra+ g \arho \la \div(\vv_h^{n+1}),y \ra   + \la \rho(\phi^*_h)g\mathbf{j}, \vv_h^{n+1}\ra = 0,
\end{align}
  where the last identity follows the definition of $\Xh$, cf. \eqref{eq:defFespace}.
\end{proof}

\begin{remark}[Finite element function spaces]
 Note that the results in Theorem \ref{thm:discreten_stable} are not restricted to this particular choice of finite element spaces. We only require that $\Vh$ and $\Qh$ are $H^1$ conforming spaces containing at least piecewise-linear functions and that the finite element pair $\Xh\times\Qh$ is an inf-sup stable for the (Navier-)Stokes equation. Furthermore, the same conclusion holds for variable time step sizes, which may be relevant for long-time simulation.  
\end{remark}

\section{Numerical results}\label{sec:numerics}

In this section we will test the proposed numerical scheme. We will consider a test case regarding phase separation in \cref{subsec:phasesep}, and convergence test in \cref{subsec:conv} and the well-known rising bubble test case in \cref{subsec:risingbubble}.  

For all test cases we use an isotropic mobility matrix of the form $\mathbf{M} = m \mathbf{I}$. Furthermore, we consider the regular potential $f(\phi)=\frac{1}{4\beta}(1-\phi^2)^2$, for some parameter $\beta$. In this case the time averaged in \eqref{eq:timeavg} can be computed exactly by
\begin{equation}
  f'(\phi_h^{n+1},\phi_h^n) = \frac{1}{6}\left(f'(\phi_h^{n+1}) + 4f'\left(\frac{\phi_h^{n+1} +\phi_h^n}{2}\right)+ f'(\phi_h^{n})   \right). 
\end{equation}

The resulting nonlinear systems are tackled by Newton's method with a tolerance $10^{-6}.$ The resulting linear systems are solved using a direct solver. The code is implemented in FEniCS \cite{Fenics}\footnote{The code is available at \href{https://github.com/marcoteneikelder/structure-preserving-fem-nsch}{https://github.com/marcoteneikelder/structure-preserving-fem-nsch}.} as well as NGSolve \cite{schoberl2014c++}\footnote{The code is available at \href{https://github.com/AaronBrunk1/structure-preserving-fem-nsch}{https://github.com/AaronBrunk1/structure-preserving-fem-nsch}.} The placeholder quantities are all evaluated at time-step $n+1$.

\subsection{Phase separation}\label{subsec:phasesep}

In this subsection we consider an test case for phase separation inspired by \cite{Gong18}. The fix our domain to be $\Omega=[0,1]^2$ with periodic boundary conditions  and denote any points $\Omega$ by $(x,y)\in\Omega$.

We consider a regular sinus shaped profile for the phase-field and no initial velocity, i.e.
\begin{equation}
   \phi_0(x,y)=0.2\sin(4\pi x)\sin(4\pi y), \qquad \vv_0(x,y) = (0,0)^\top.
\end{equation}

We consider the parameter choices

\begin{equation*}
 m(\phi) = 10^{-2}(1-\phi^2)^2,\quad \beta = \gamma = 10^{-3/2},\quad \eta_1=\eta_2  = 10^{-2},\quad g=0.
\end{equation*}

The above experiment was conducted for different density ratios, i.e. $\rho_1:\rho_2\in\{10^0:10^3,10^0:10^2,10^0:10^1,10^1:10^0,10^2:10^0,10^3:10^0\}$.

In Figure \ref{fig:evophic} we illustrate the temporal evolution for the phase-field $\phi$ at several snapshots in time and the associated energy evolution and mass conservation error in Figure \ref{fig: meta}. First consider the density ratio is $10:1$, fluid phase 1 forms droplets within the other fluid phase. These droplets form rapidly, causing the energy to decay quickly to a saturated level. Similarly, when the density ratio is reversed to 1:10, the heavier fluid component forms droplets within the lighter fluid component. After a quite long saturation phase the matrix pattern connects into strips, again lowering the energy. In both scenarios, the heavier fluid component appears as droplets within the lighter fluid, demonstrating a symmetrical morphology. When increasing the density ratio we can observe a rescaling in time. The morphology effects of droplet formation for the ratios $10:1$ and $1:10$ is observed at $t=0.1$, while for the ratios $100:1$ and $1:100$ at $t=0.3$ and finally at time $t=1$ for the ratios $1000:1$ and $1:1000$.

	\begin{figure}[htbp!]
		\centering
		\footnotesize
		\begin{tabular}{c@{}c@{}c@{}c}
			\includegraphics[trim={38cm 12.4cm 38.cm 8.5cm},clip,scale=0.057]{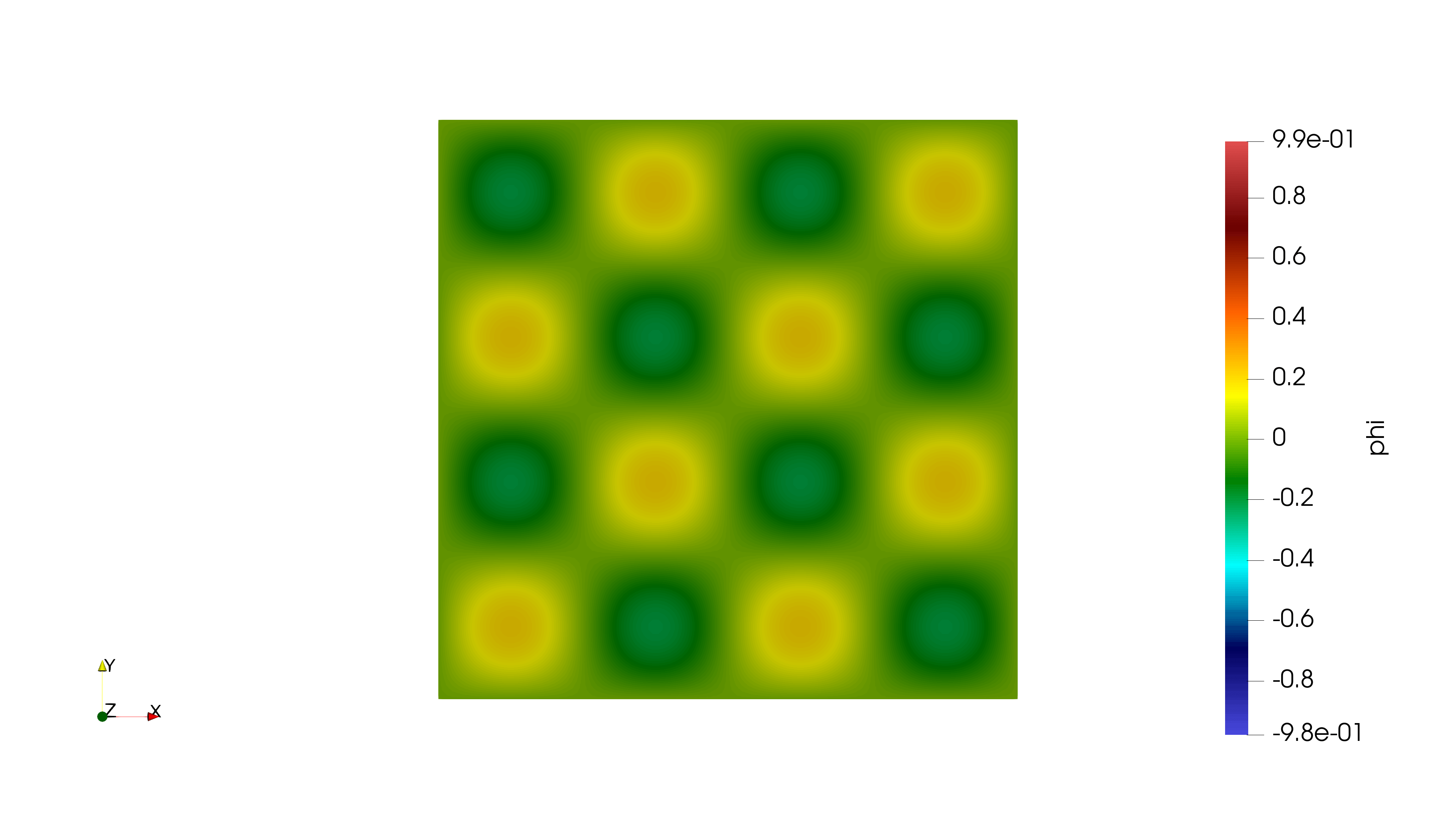} 
			&
			\includegraphics[trim={38cm 12.4cm 38.cm 8.5cm},clip,scale=0.057]{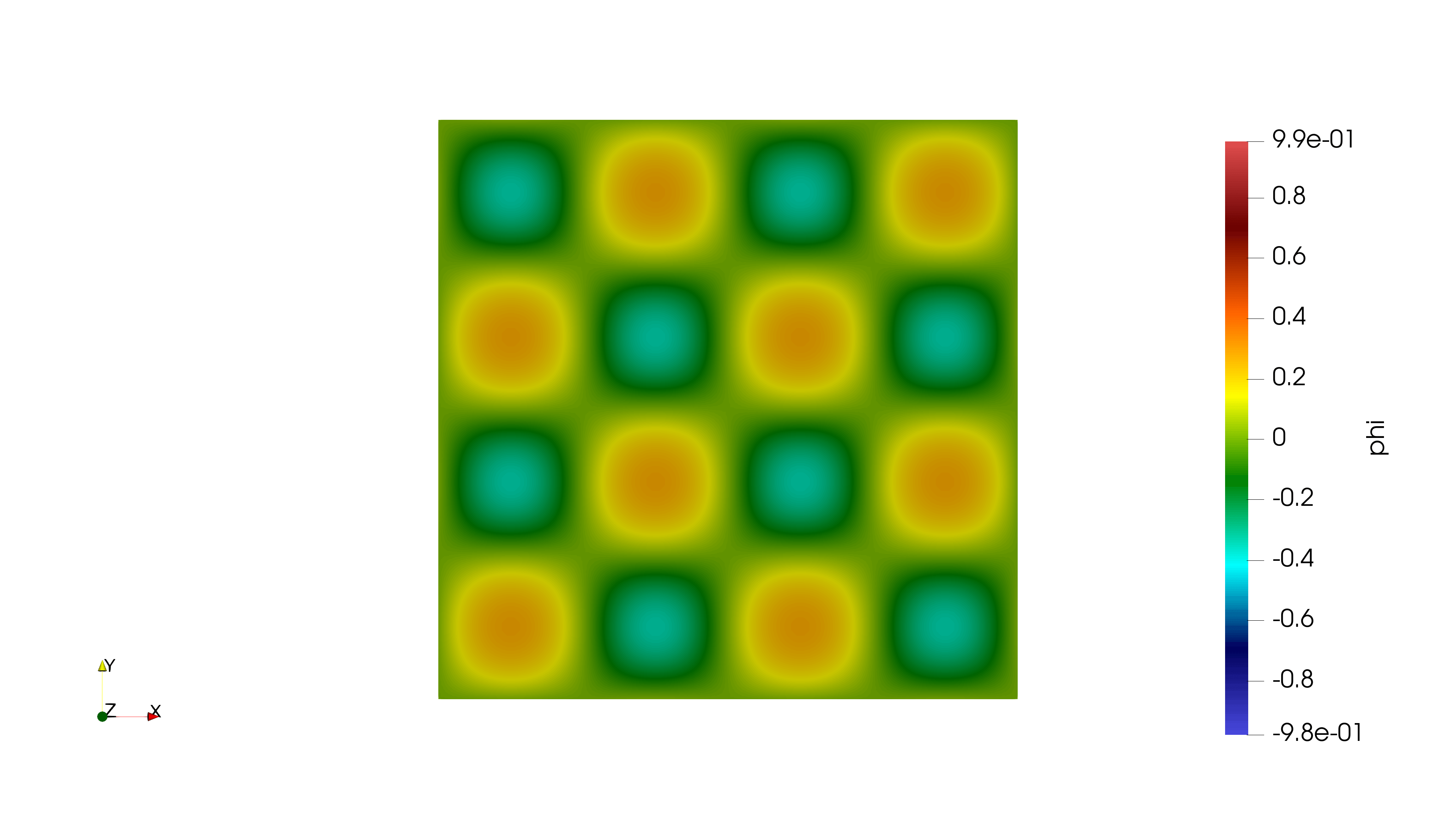}  
			&
			\includegraphics[trim={38cm 12.4cm 38.cm 8.5cm},clip,scale=0.057]{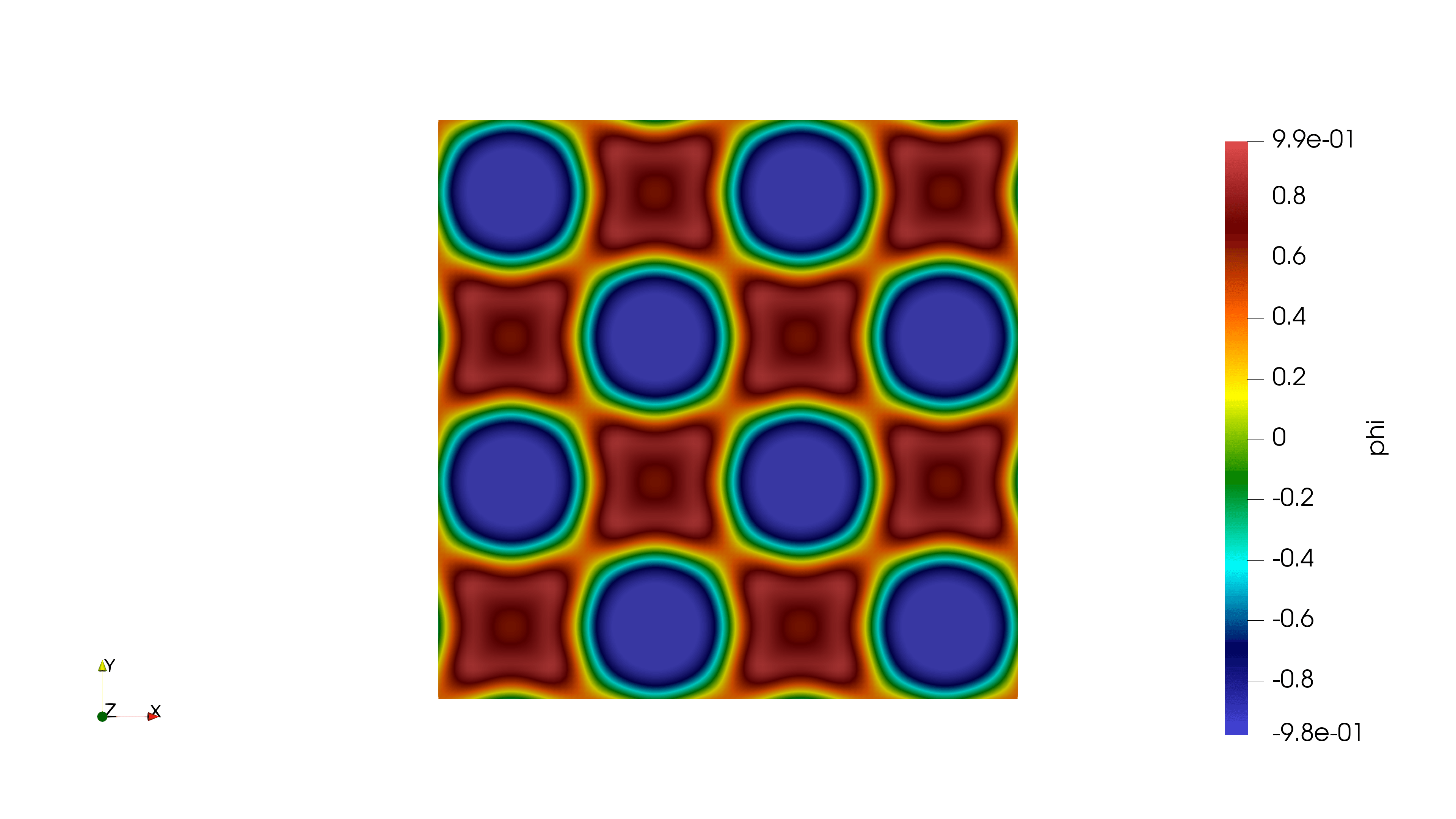} 
			&
			\includegraphics[trim={38cm 12.4cm 38.cm 8.5cm},clip,scale=0.057]{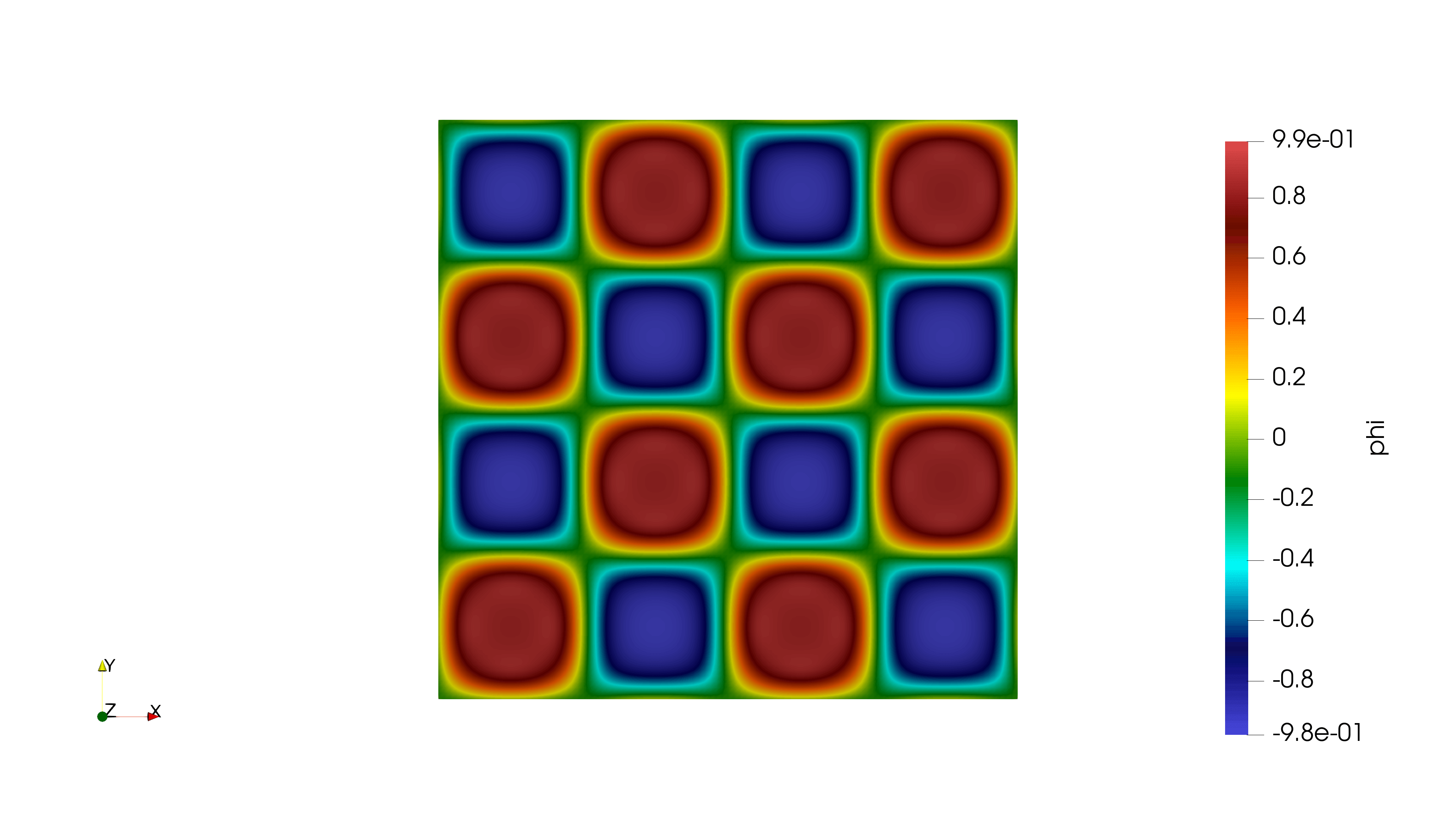}  \\
			\includegraphics[trim={38cm 12.4cm 38.cm 8.5cm},clip,scale=0.055]{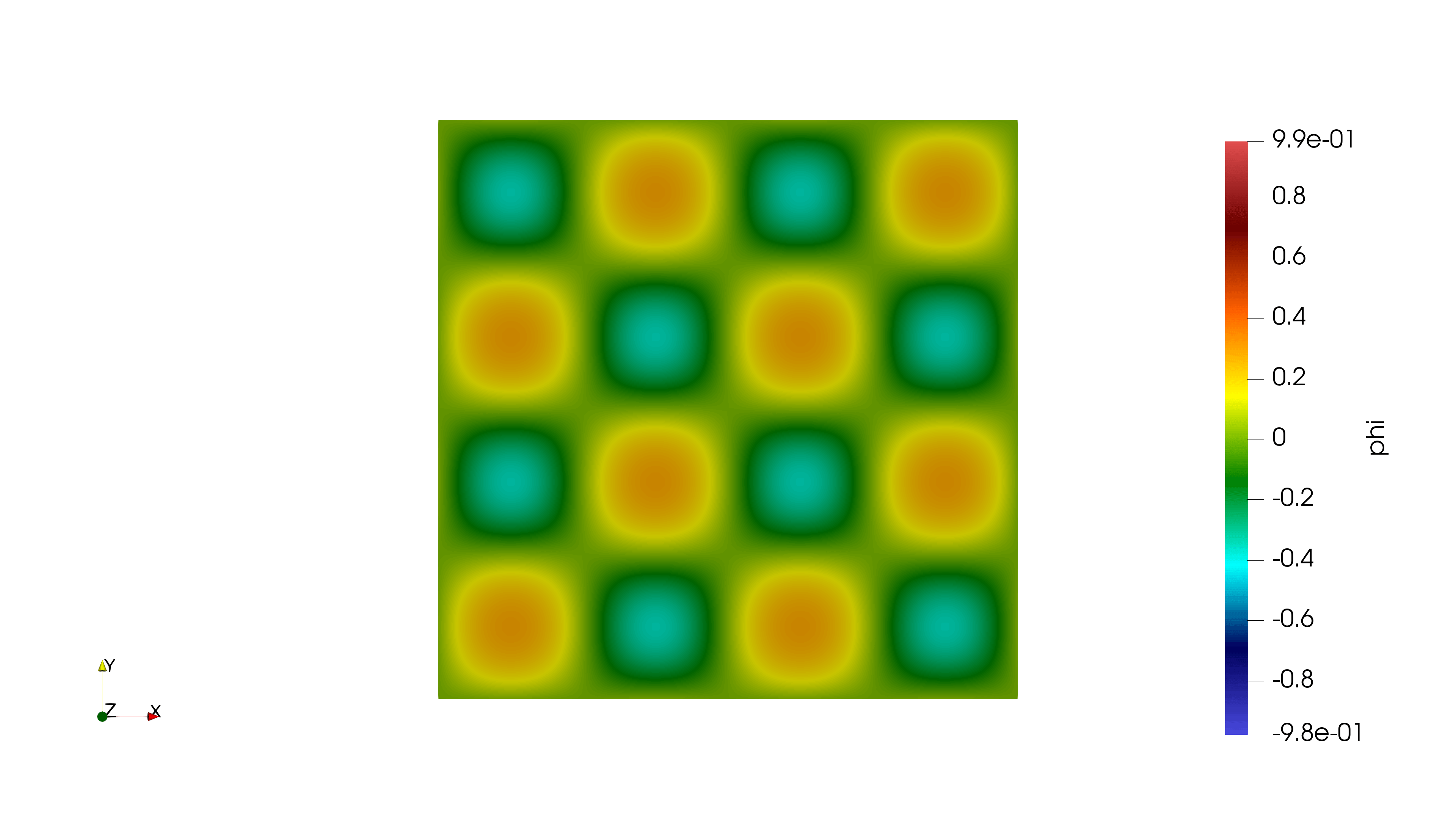} 
			&
			\includegraphics[trim={38cm 12.4cm 38.cm 8.5cm},clip,scale=0.055]{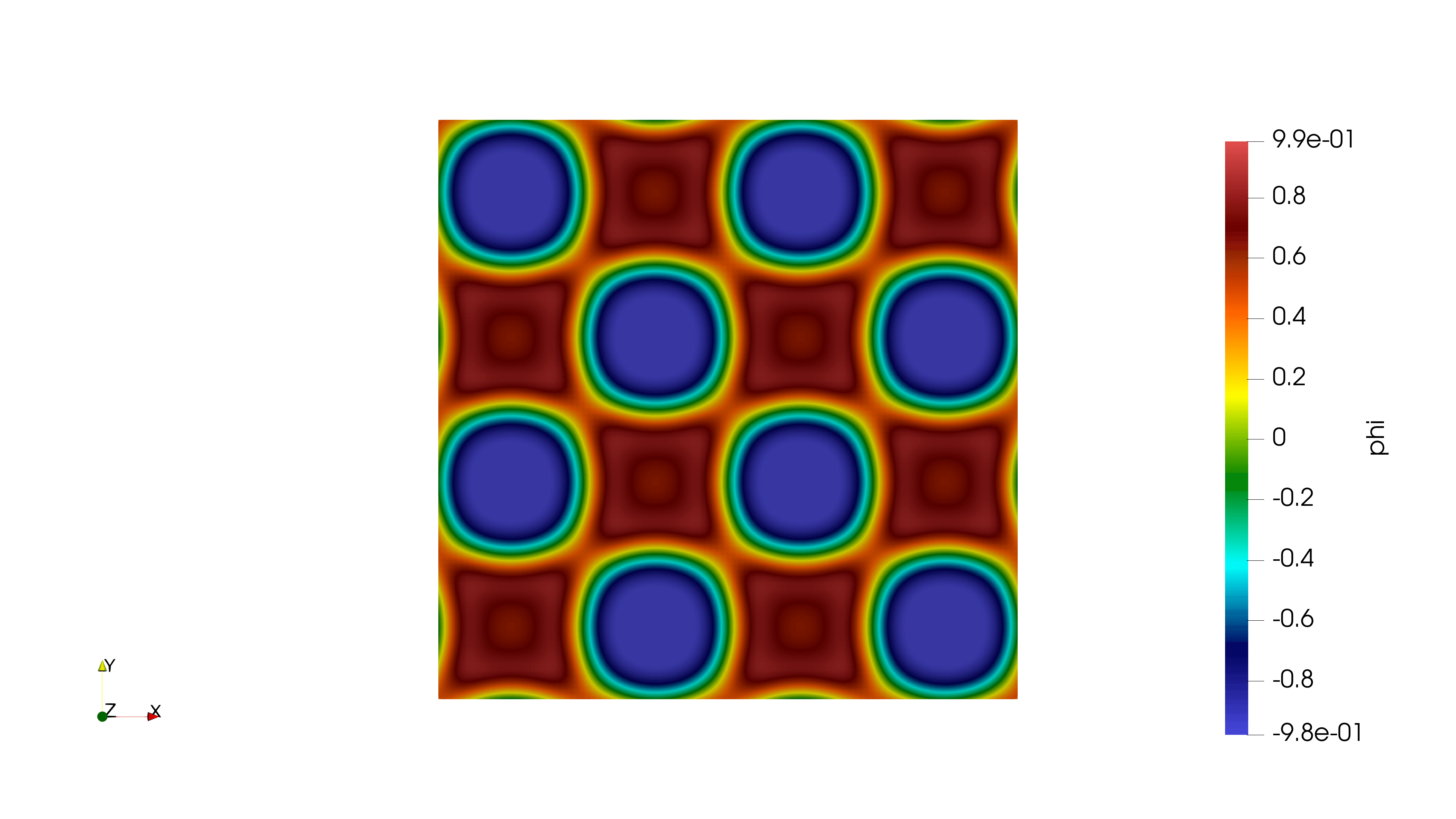}  
			&
			\includegraphics[trim={38cm 12.4cm 38.cm 8.5cm},clip,scale=0.055]{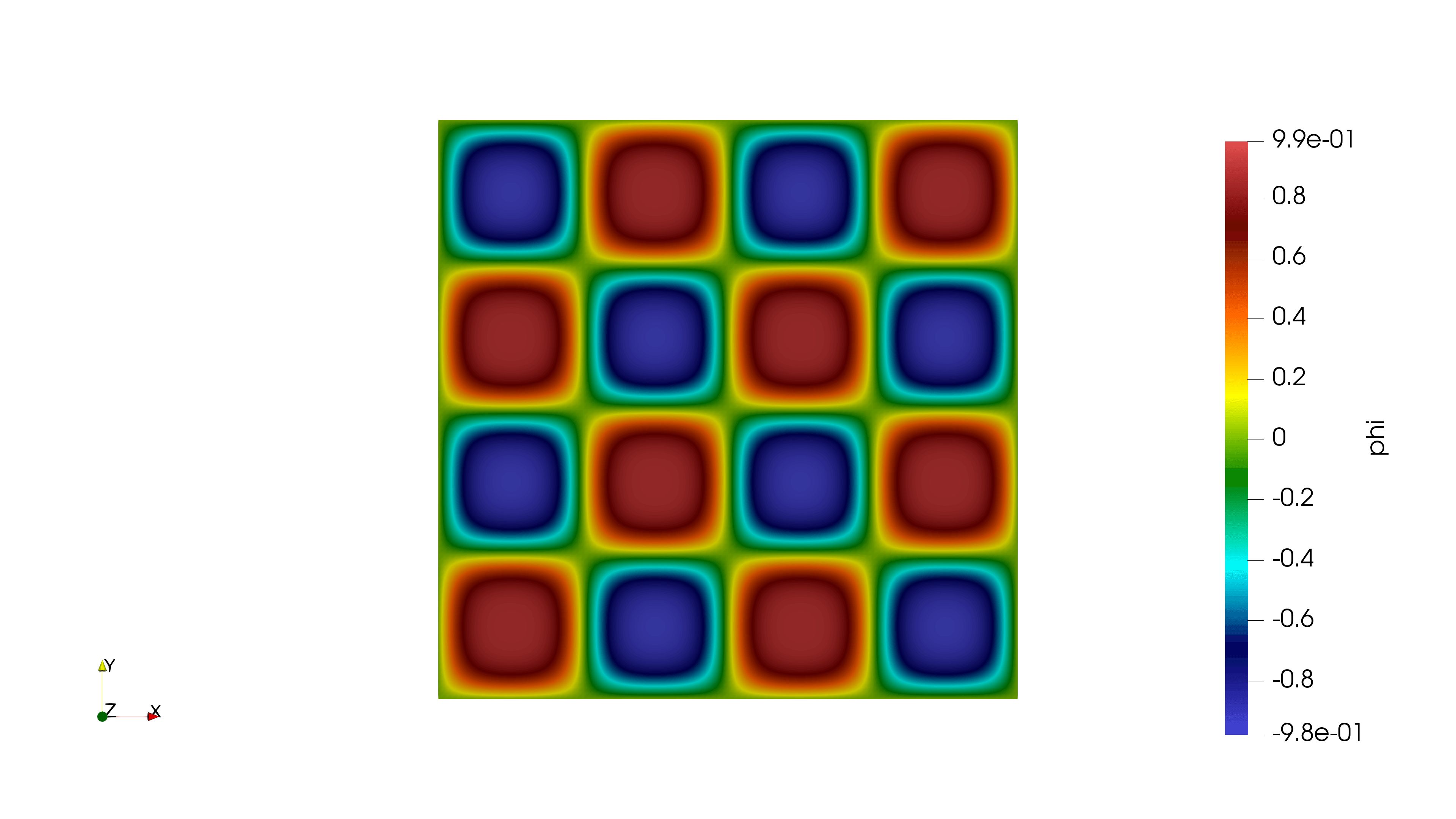} 
			&
			\includegraphics[trim={38cm 12.4cm 38.cm 8.5cm},clip,scale=0.055]{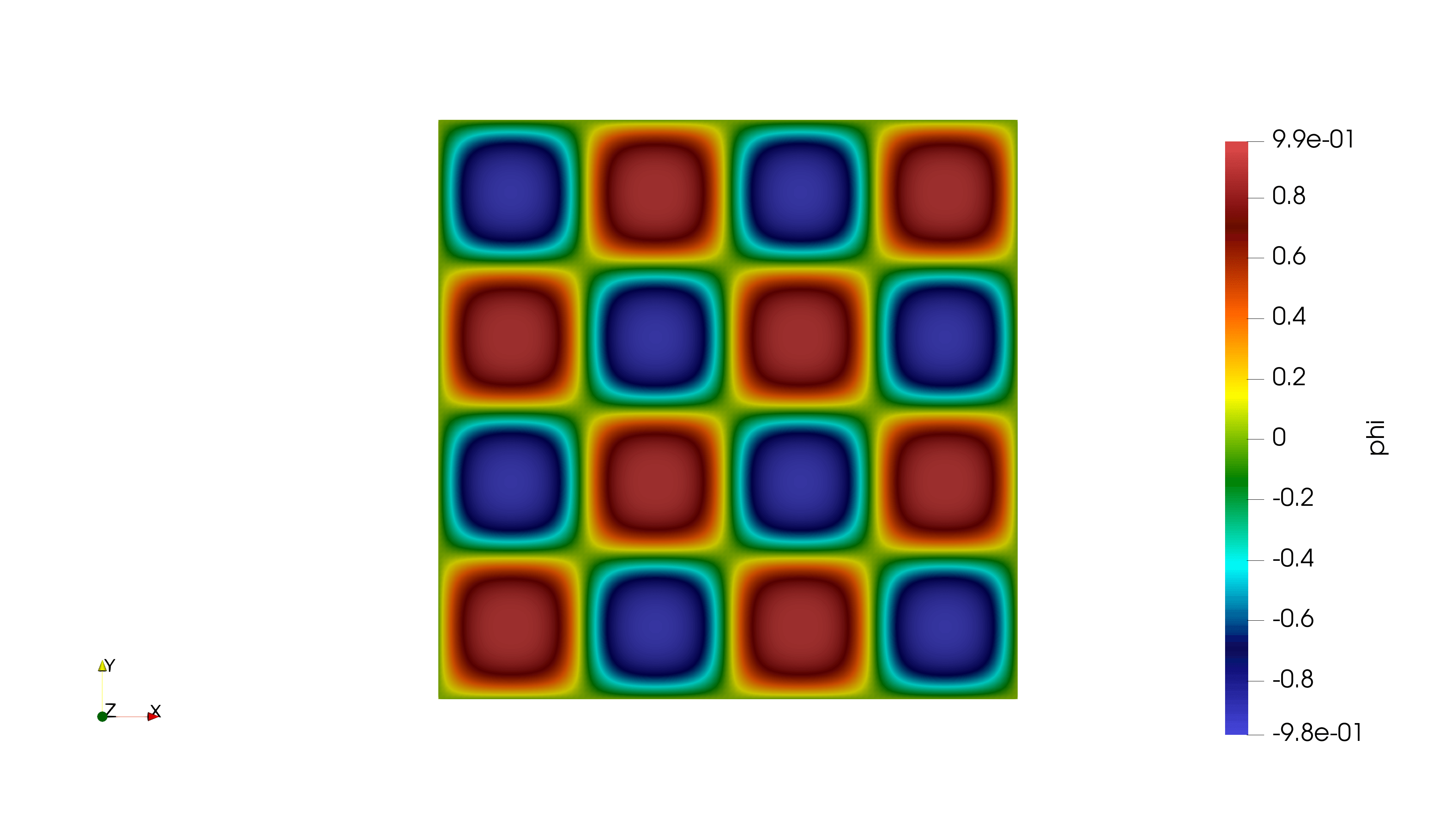}  \\
			\includegraphics[trim={38cm 12.4cm 38.cm 8.5cm},clip,scale=0.055]{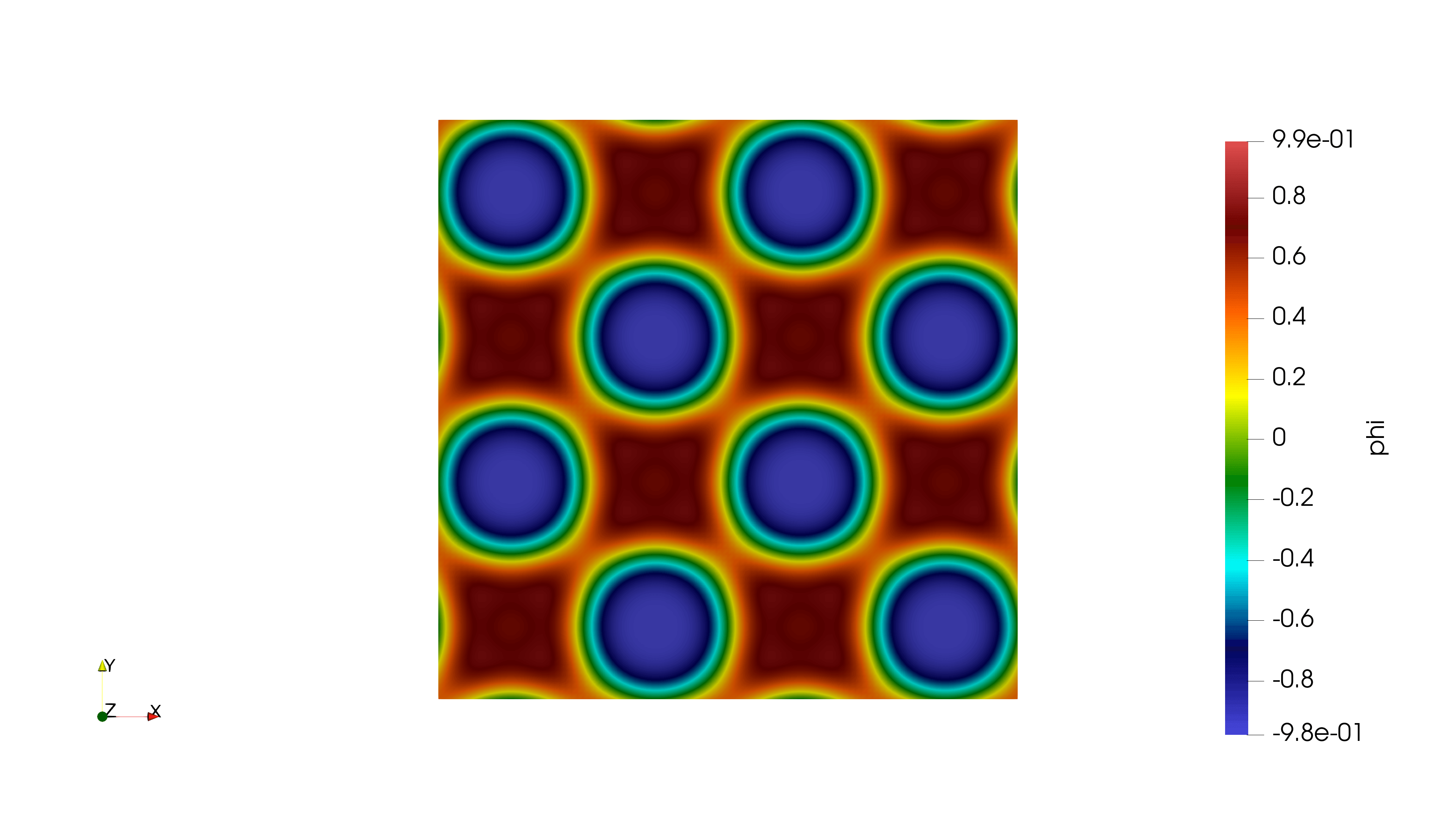} 
			&
			\includegraphics[trim={38cm 12.4cm 38.cm 8.5cm},clip,scale=0.055]{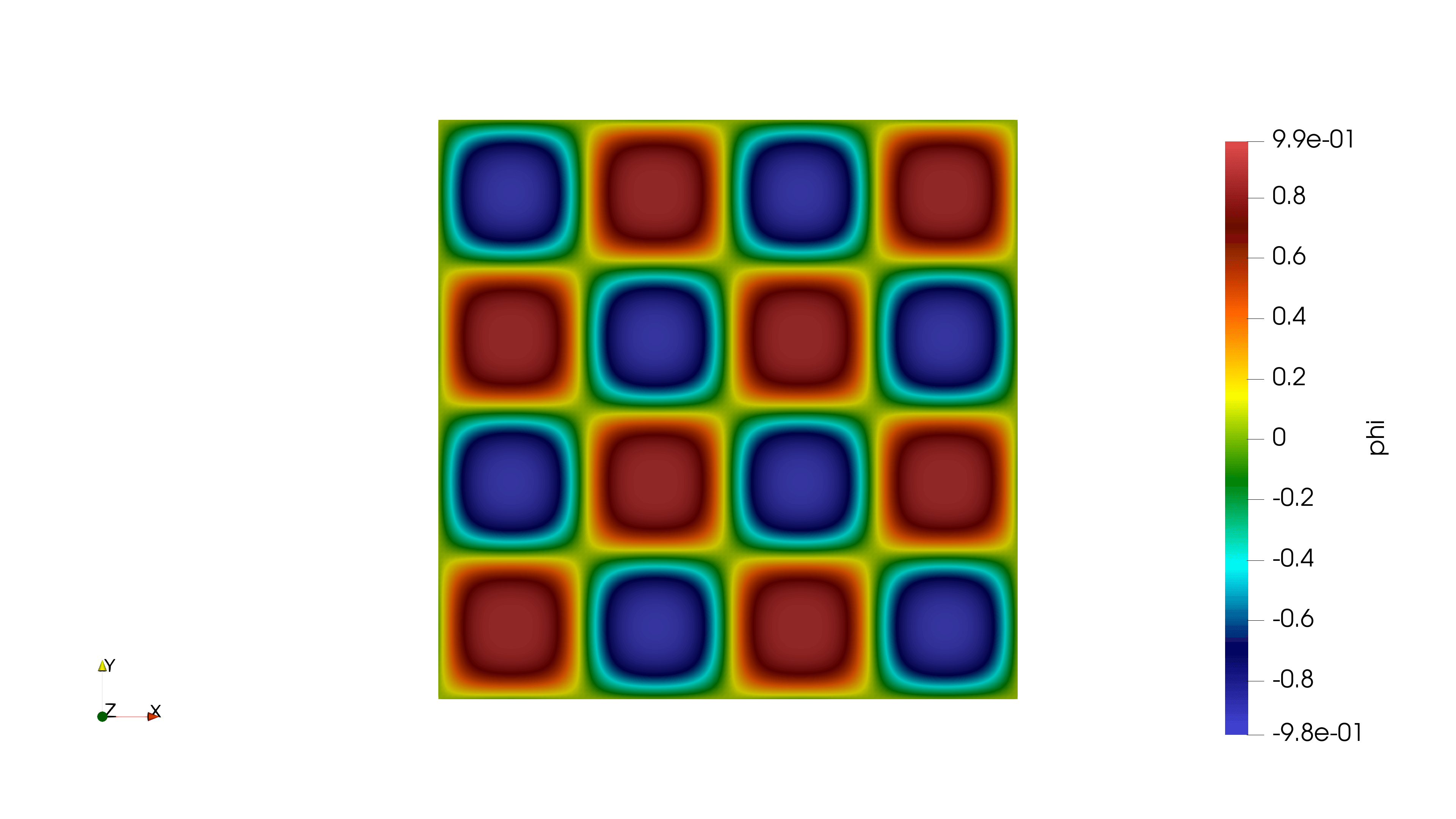}  
			&
			\includegraphics[trim={38cm 12.4cm 38.cm 8.5cm},clip,scale=0.055]{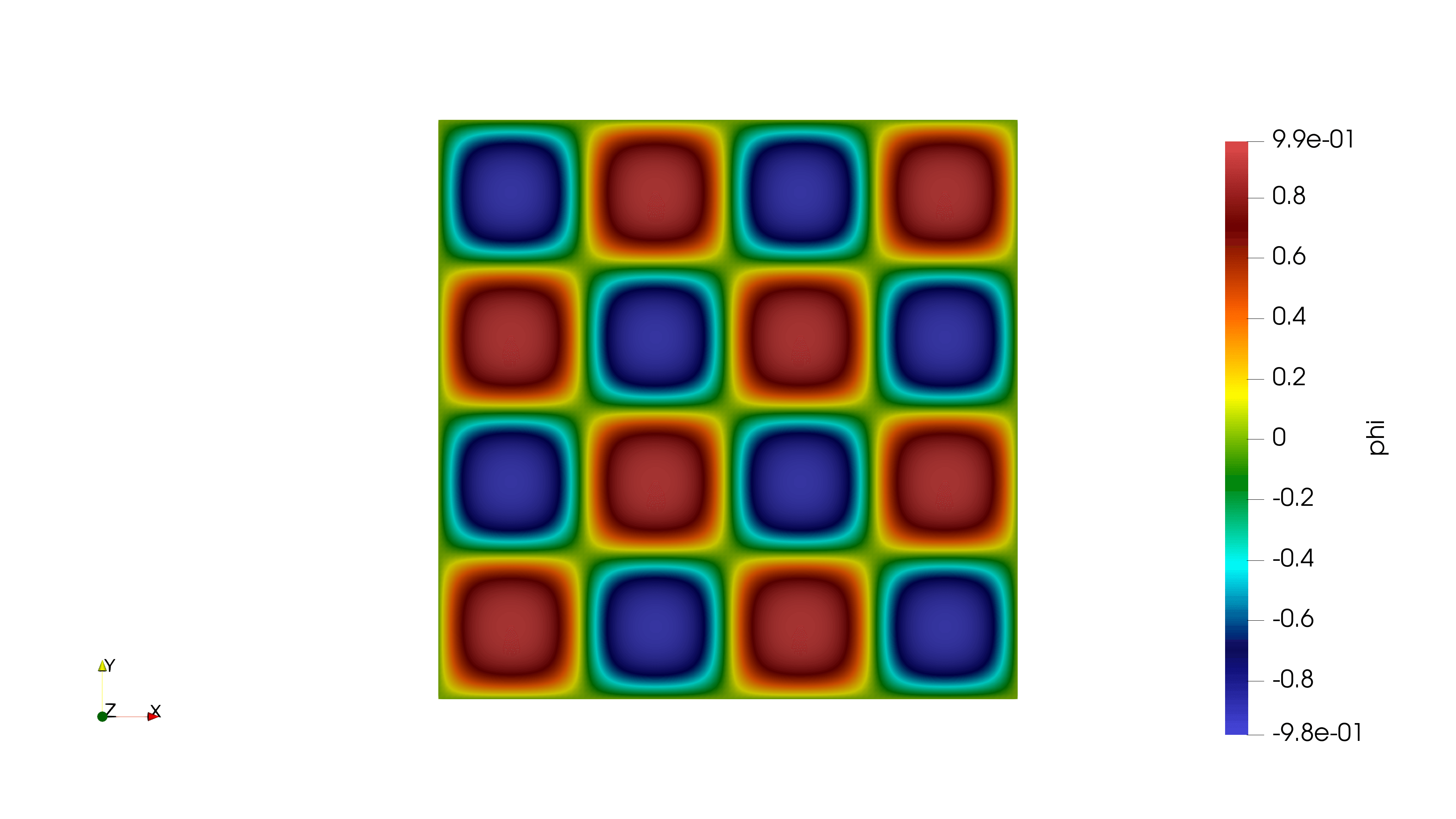} 
			&
			\includegraphics[trim={38cm 12.4cm 38.cm 8.5cm},clip,scale=0.055]{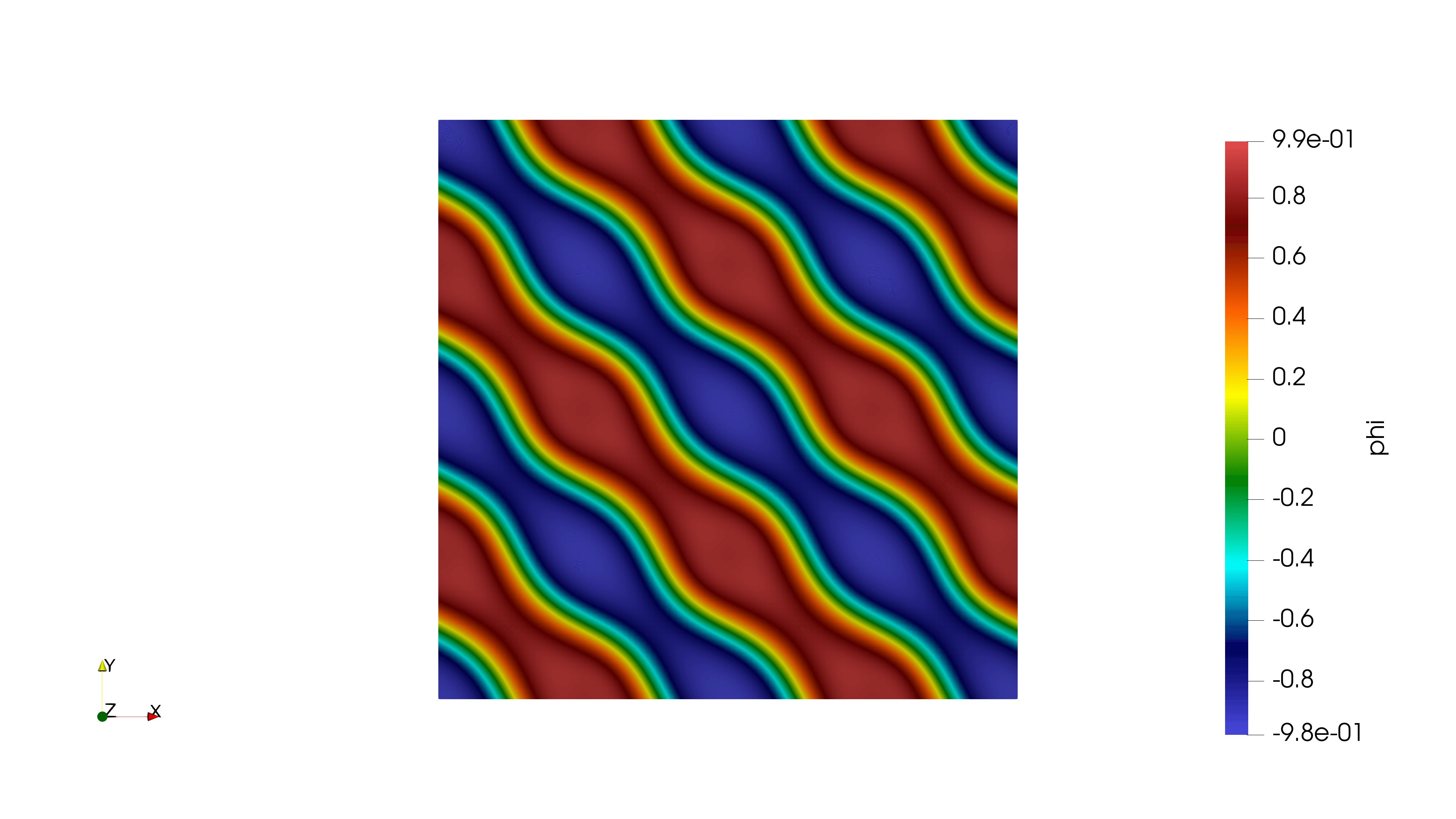} \\
			\includegraphics[trim={38cm 12.4cm 38.cm 8.5cm},clip,scale=0.055]{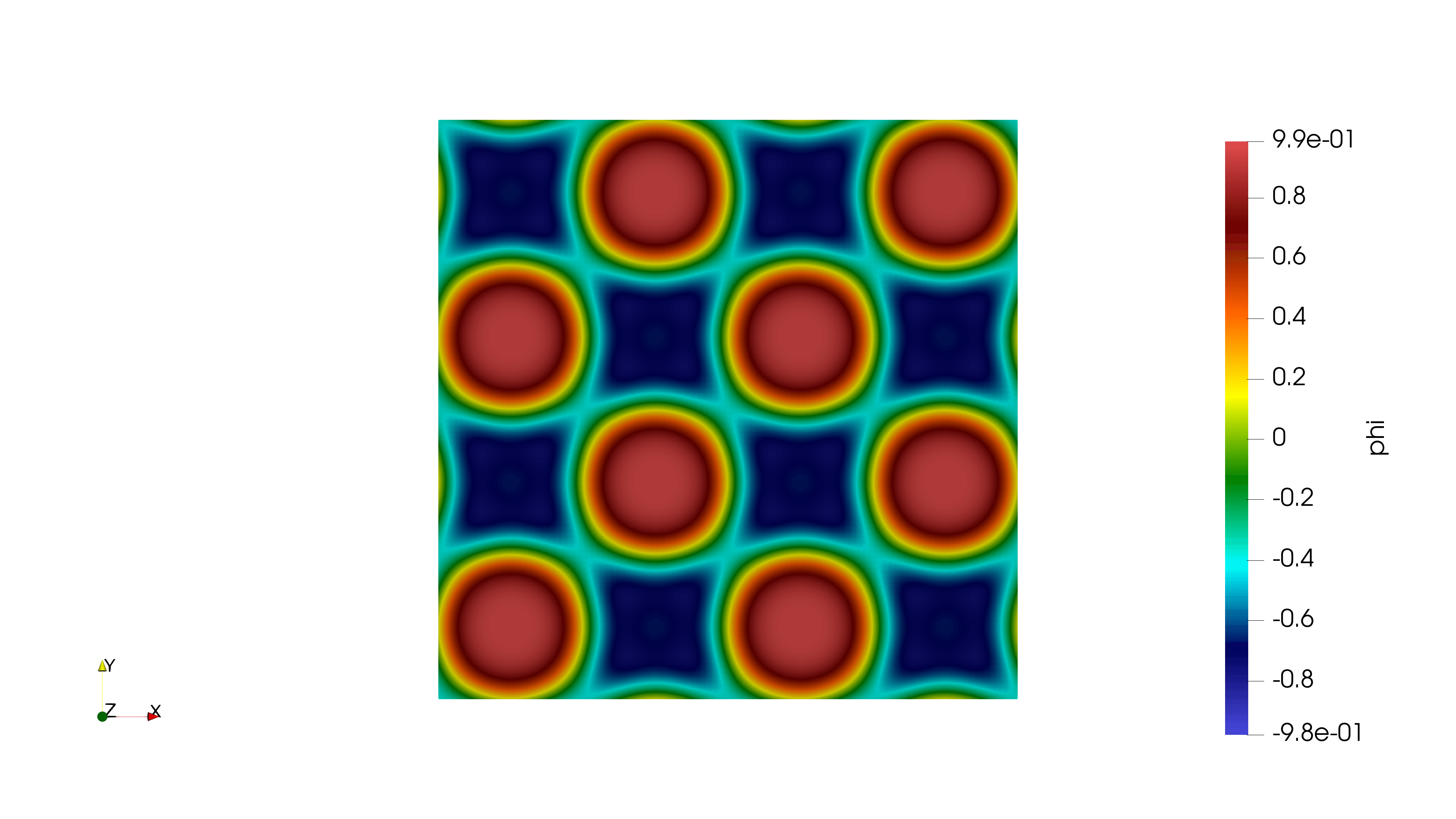} 
			&
			\includegraphics[trim={38cm 12.4cm 38.cm 8.5cm},clip,scale=0.055]{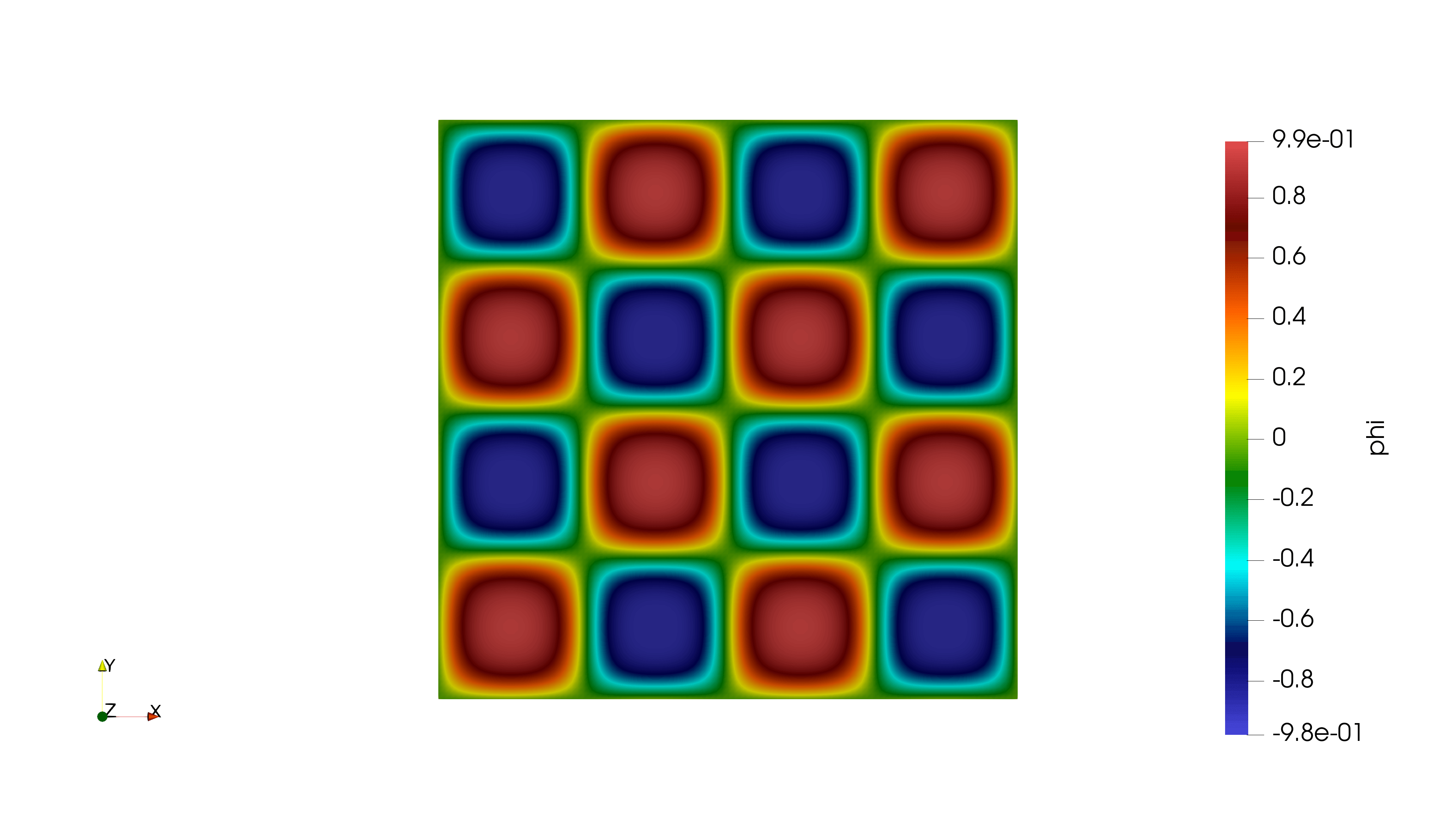}  
			&
			\includegraphics[trim={38cm 12.4cm 38.cm 8.5cm},clip,scale=0.055]{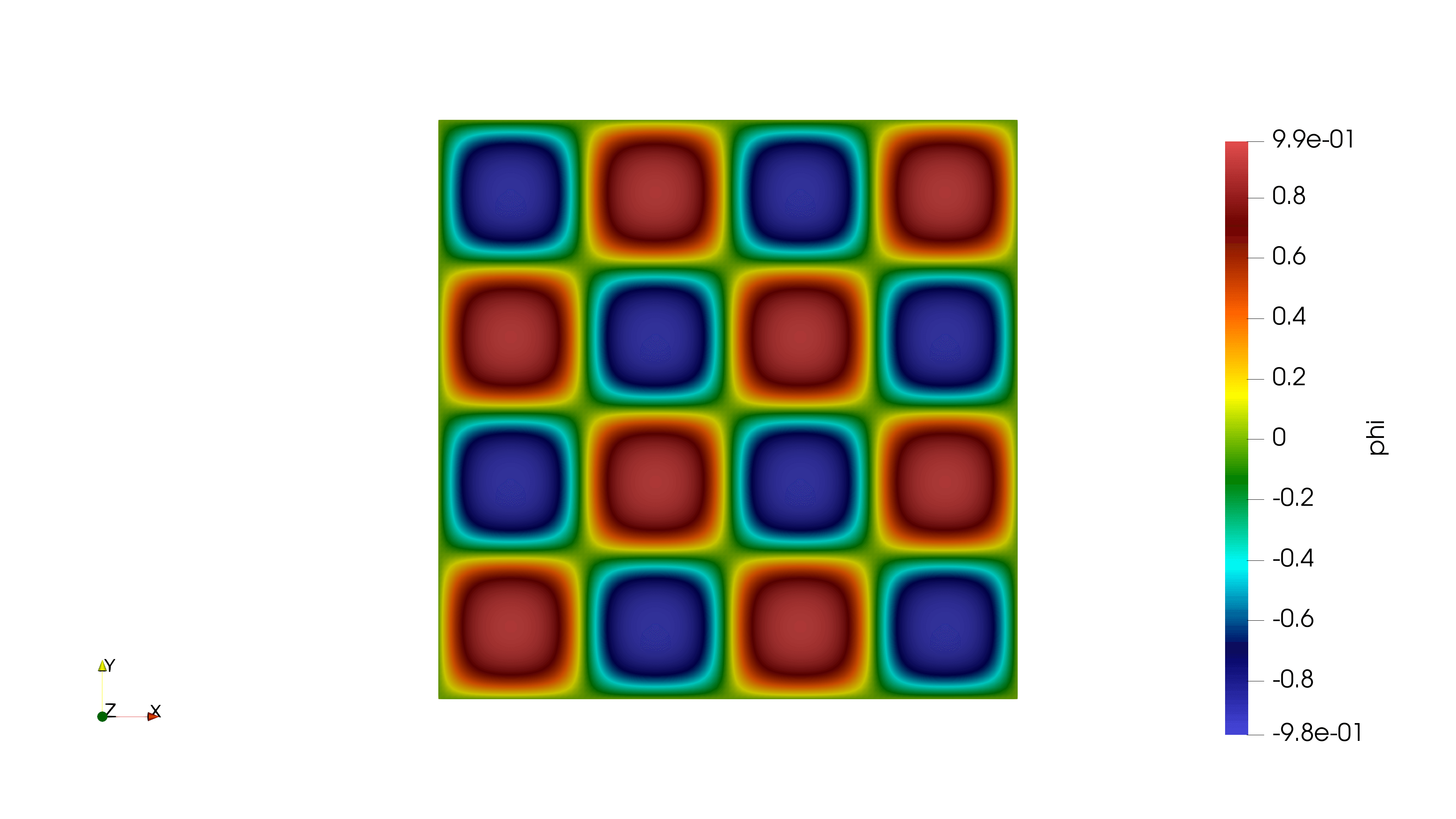} 
			&
			\includegraphics[trim={38cm 12.4cm 38.cm 8.5cm},clip,scale=0.055]{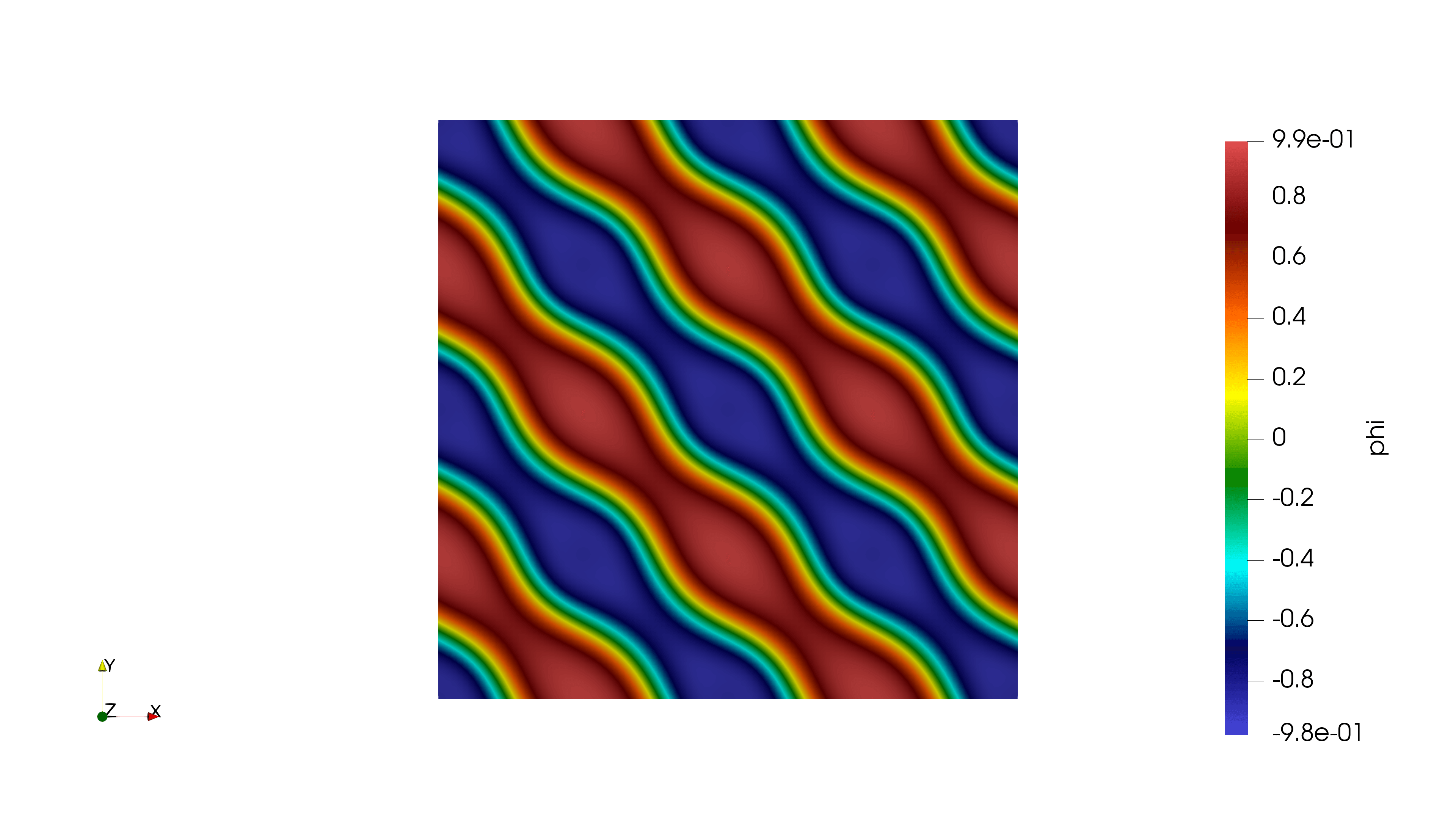}  \\
			\includegraphics[trim={38cm 12.4cm 38.cm 8.5cm},clip,scale=0.055]{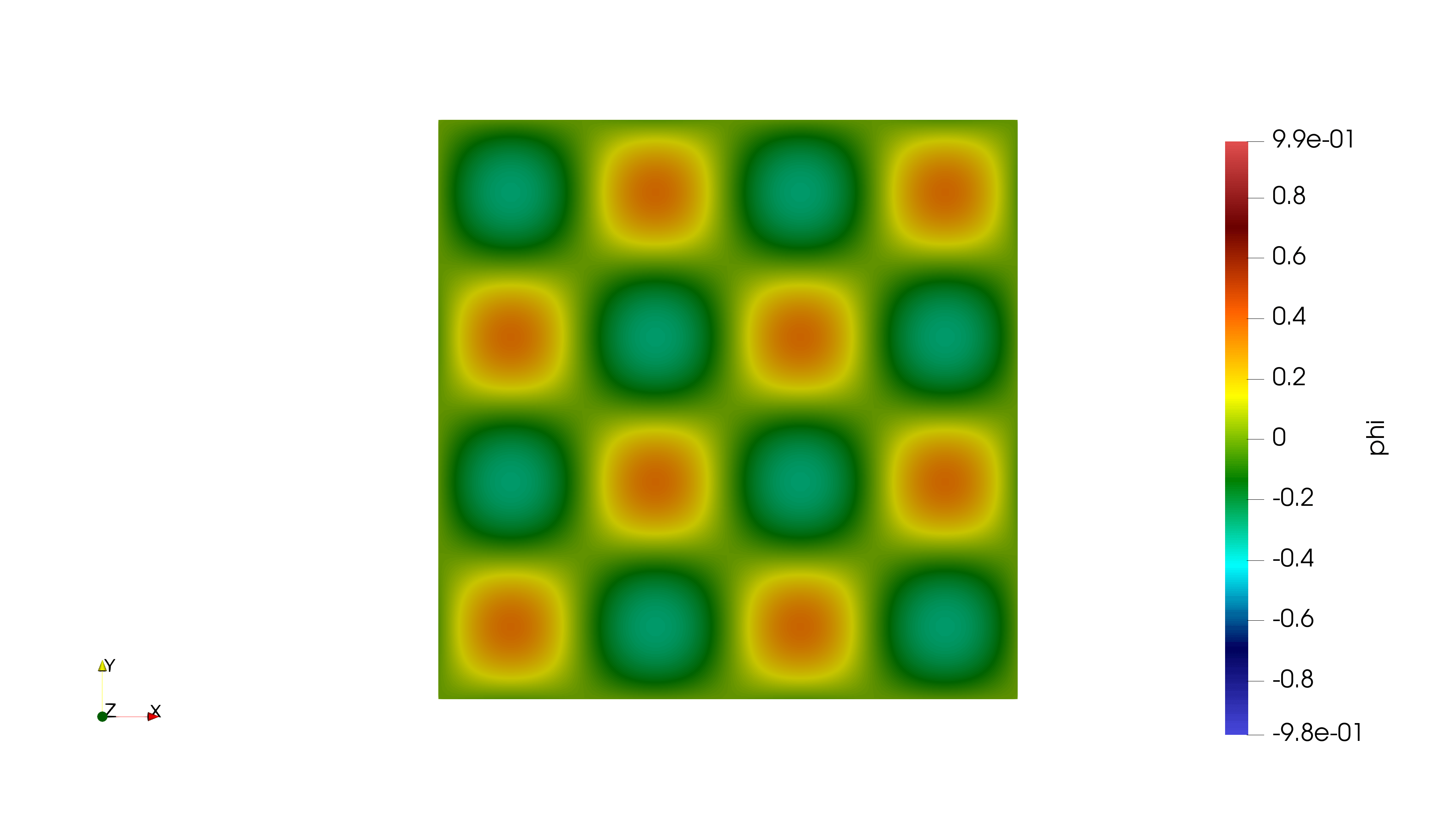} 
			&
			\includegraphics[trim={38cm 12.4cm 38.cm 8.5cm},clip,scale=0.055]{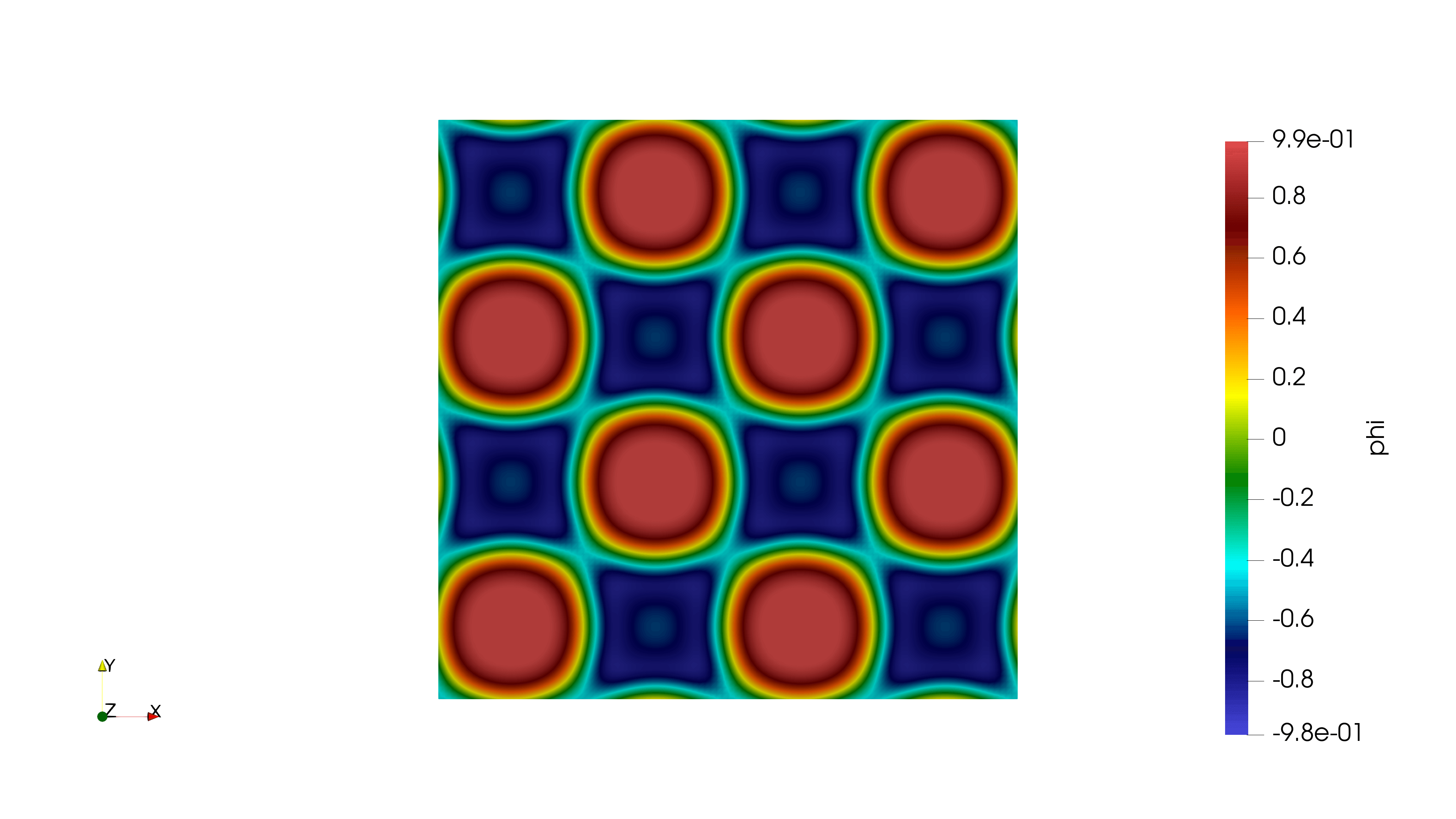}  
			&
			\includegraphics[trim={38cm 12.4cm 38.cm 8.5cm},clip,scale=0.055]{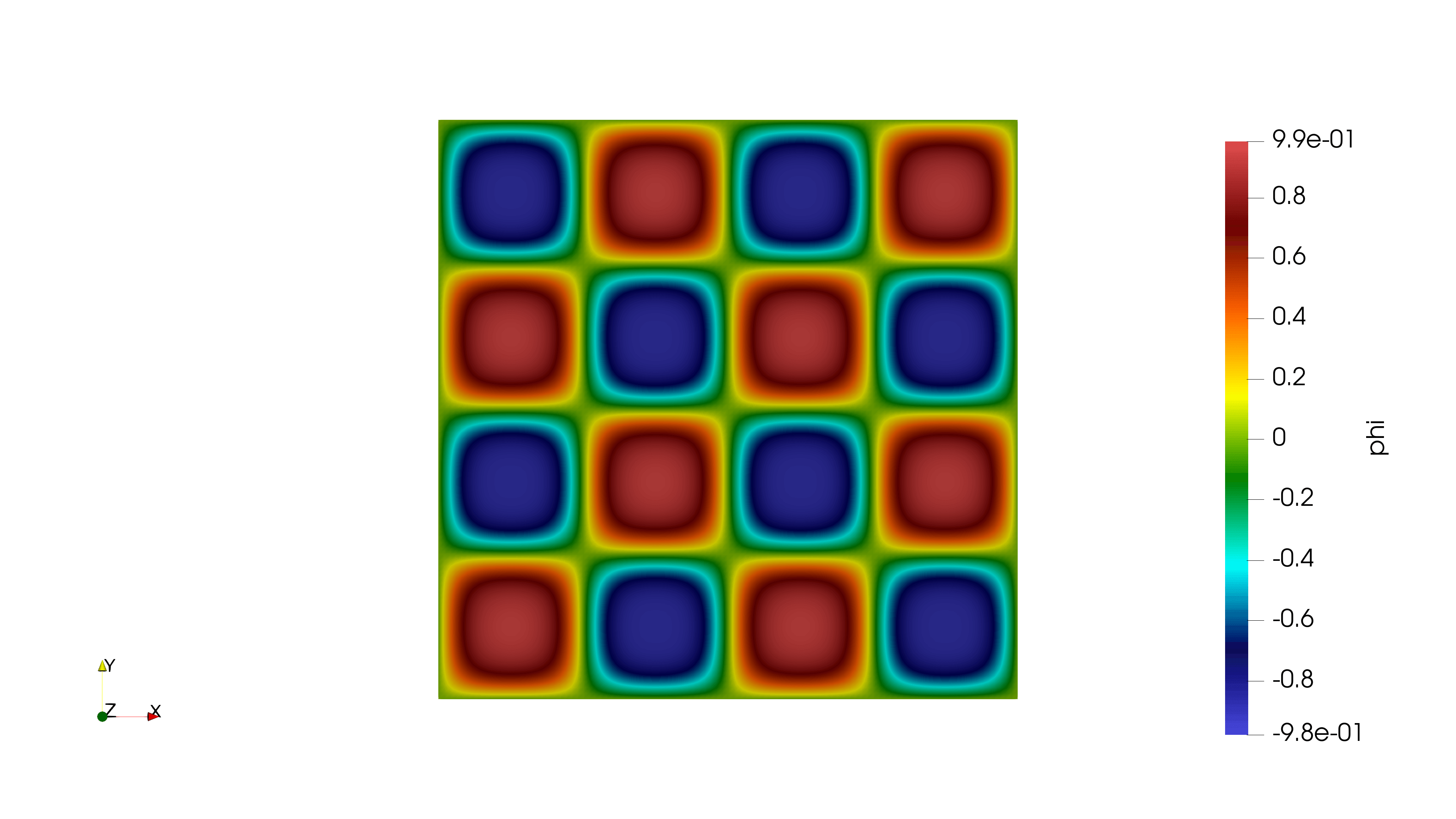} 
			&
			\includegraphics[trim={38cm 12.4cm 38.cm 8.5cm},clip,scale=0.055]{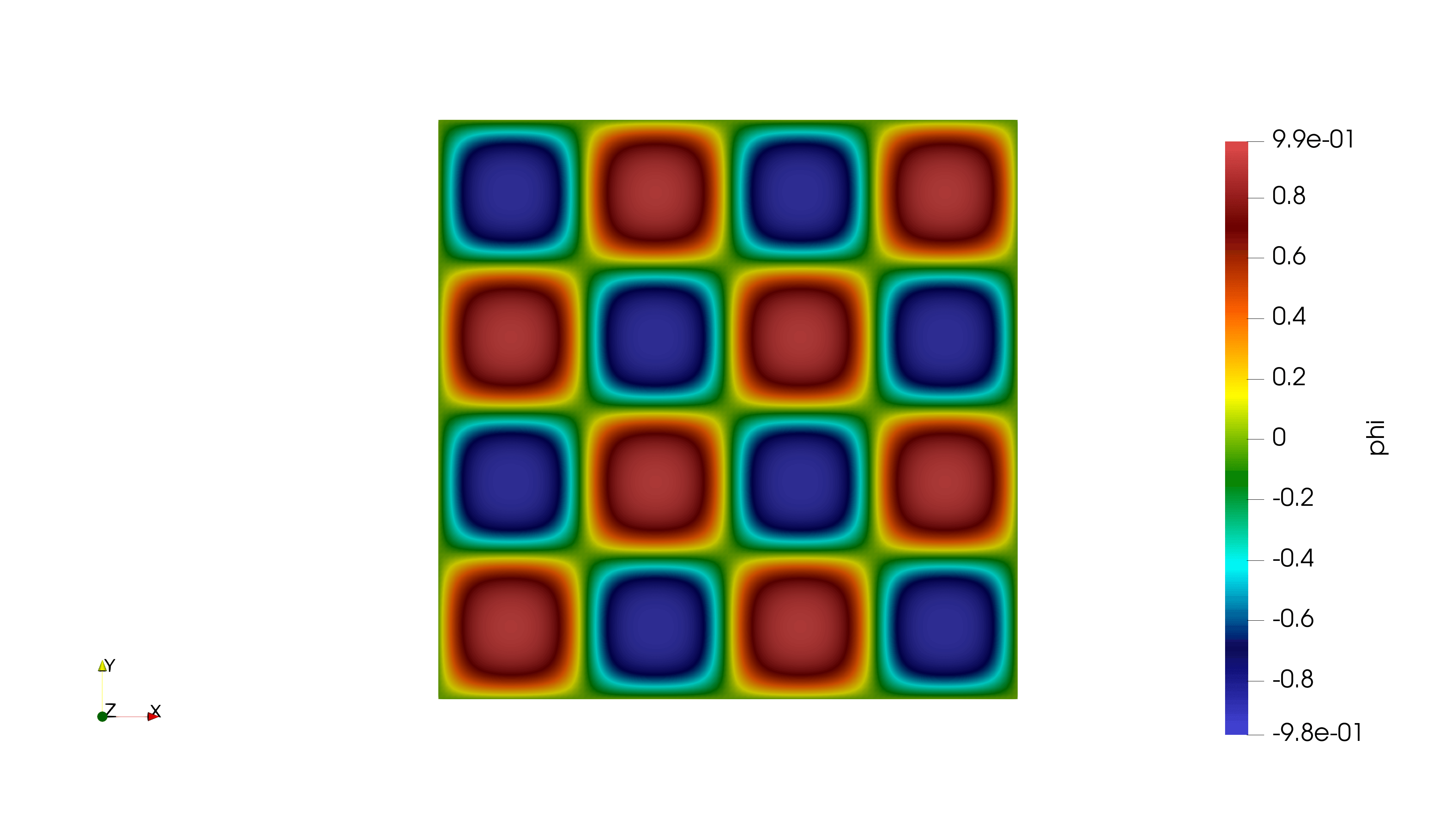}  \\
			\includegraphics[trim={38cm 12.4cm 38.cm 8.5cm},clip,scale=0.055]{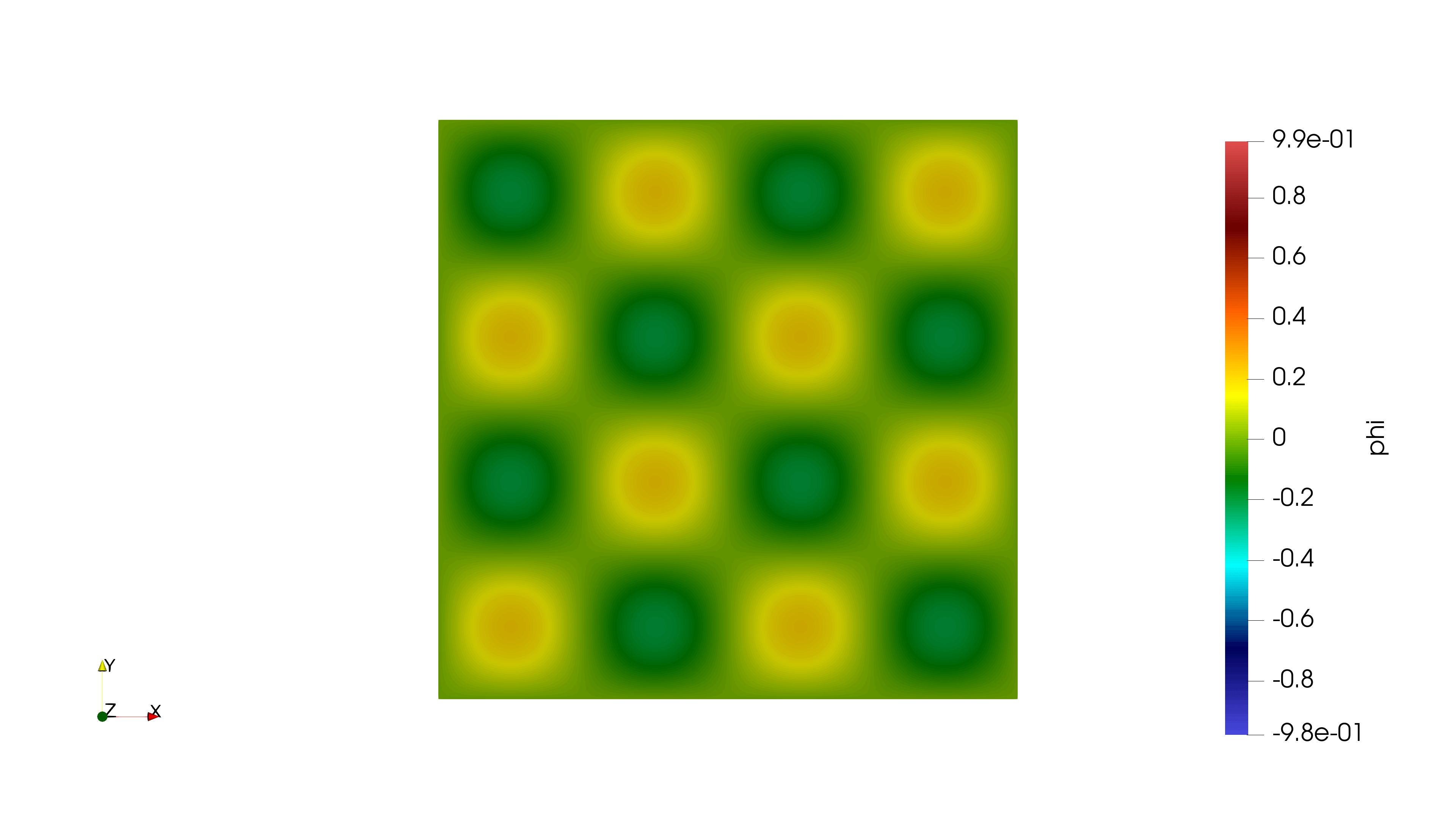} 
			&
			\includegraphics[trim={38cm 12.4cm 38.cm 8.5cm},clip,scale=0.055]{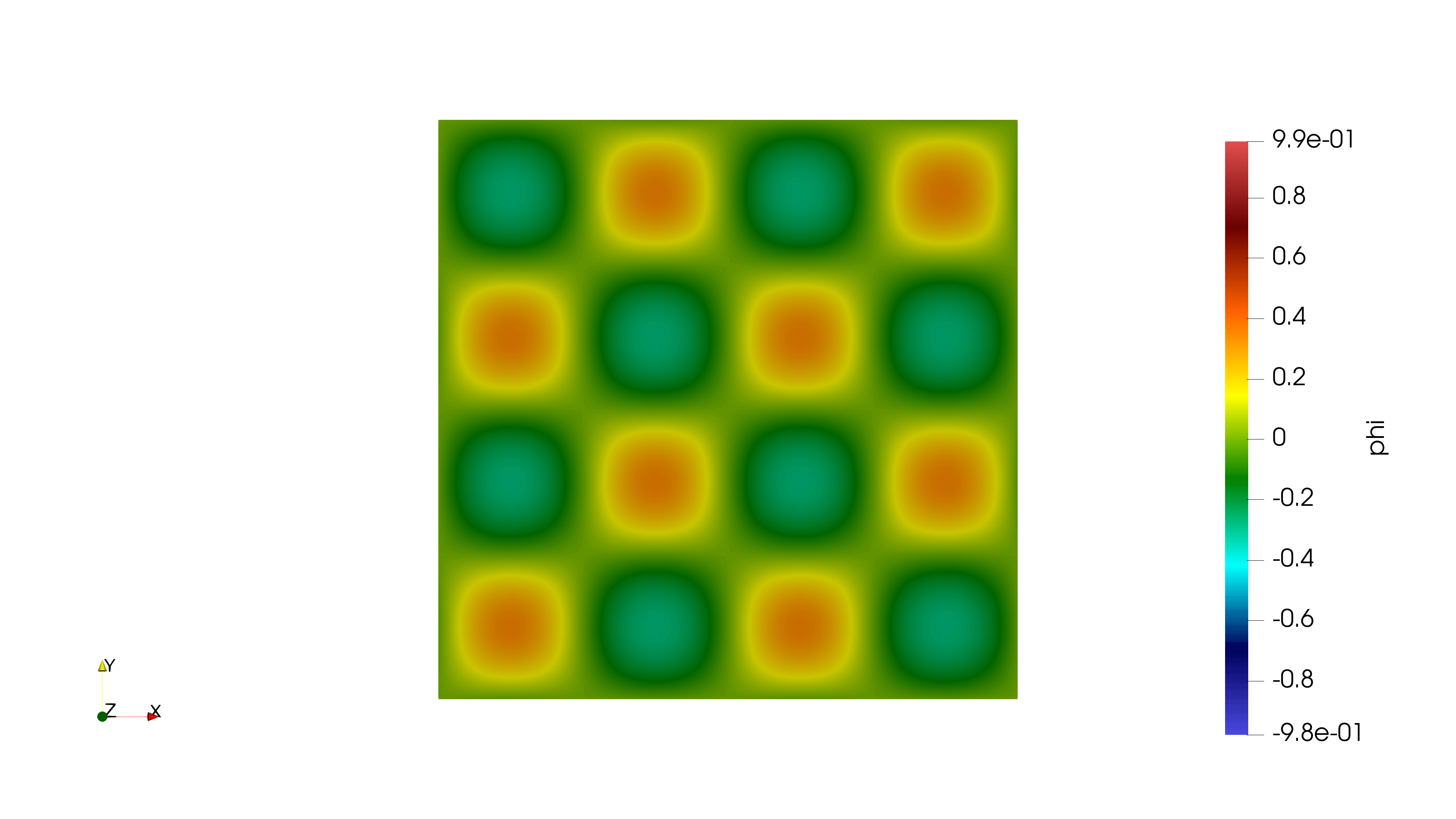}  
			&
			\includegraphics[trim={38cm 12.4cm 38.cm 8.5cm},clip,scale=0.055]{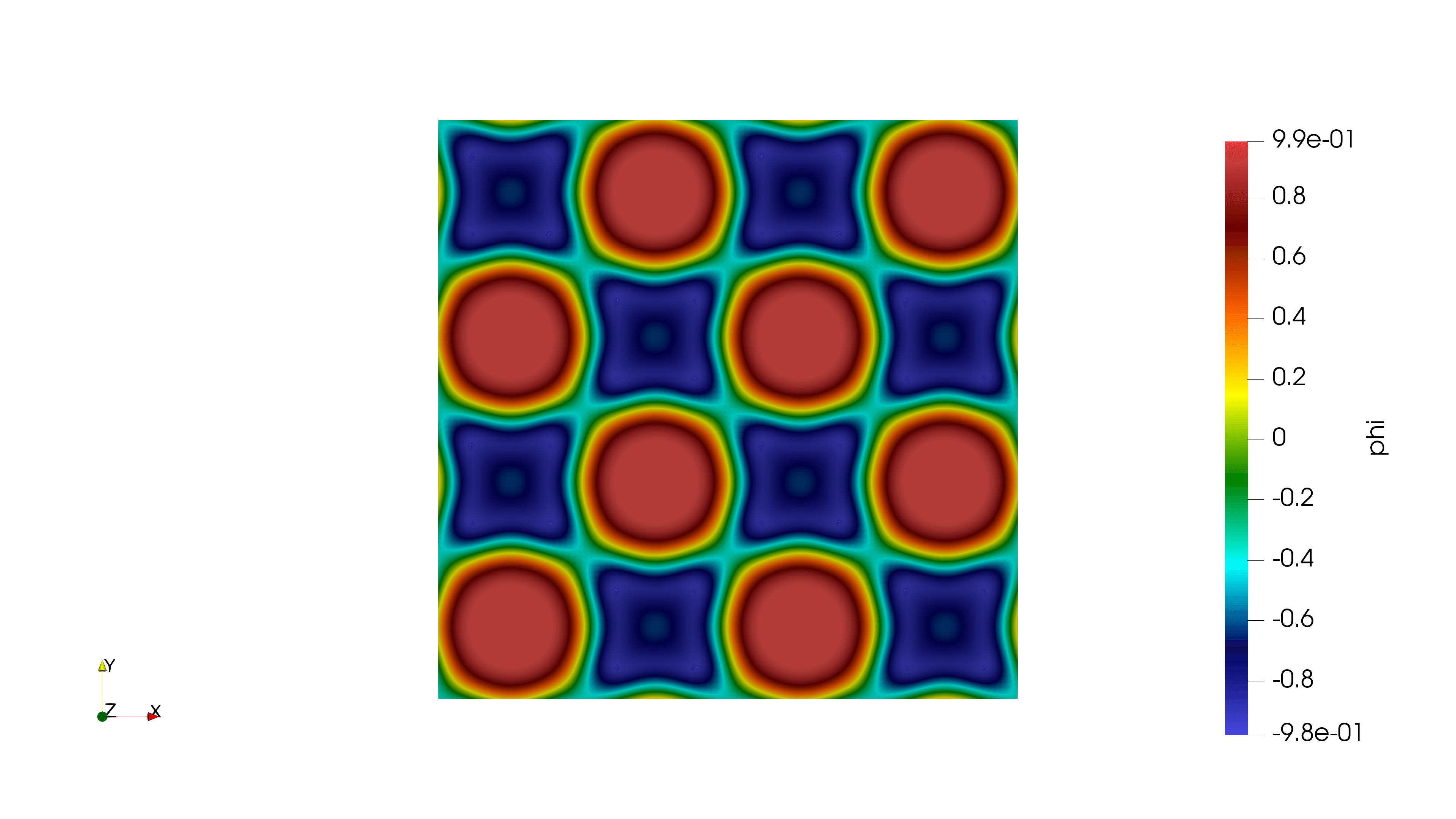} 
			&
			\includegraphics[trim={38cm 12.4cm 38.cm 8.5cm},clip,scale=0.055]{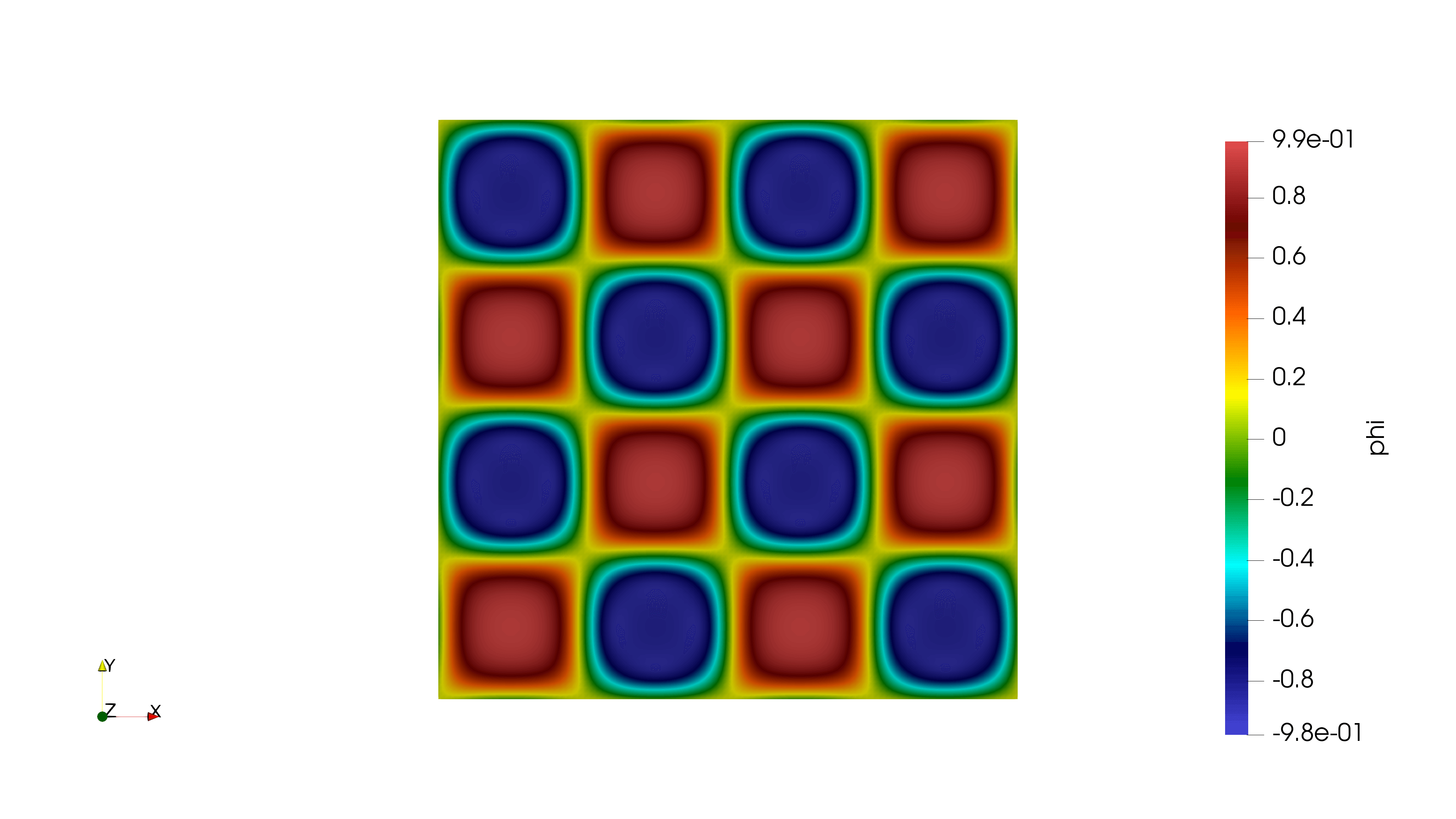} 
		\end{tabular}
		\caption{Phase separation: Snapshots of the volume fraction $\phi$ for the density ratios $\rho_1:\rho_2\in\{10^0:10^3,10^0:10^2,10^0:10^1,10^1:10^0,10^2:10^0,10^3:10^0\}$ from top to bottom at the times $\{0.1,0.3,1,2\}$ from left to right.  \label{fig:evophic}}
	\end{figure}

    \begin{figure}[!ht]
    \begin{subfigure}{0.49\textwidth}
    \centering
    \includegraphics[trim={4cm 18.4cm 5cm 2.5cm},clip,width=0.95\textwidth]{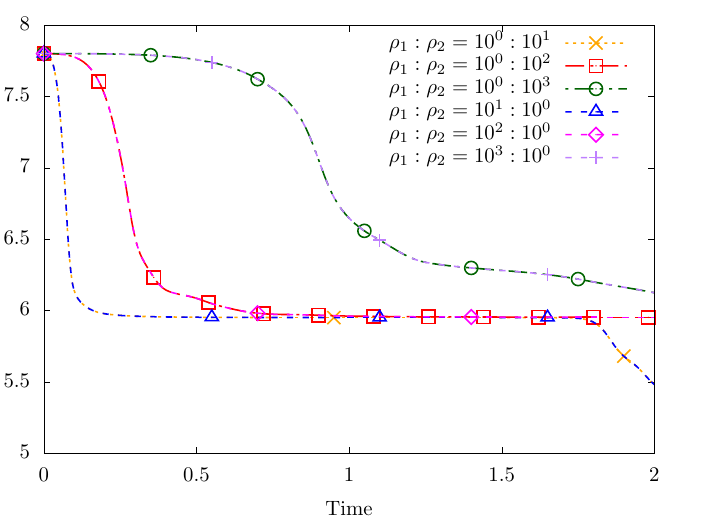}
    \end{subfigure}
    \begin{subfigure}{0.49\textwidth}
    \centering
    \includegraphics[trim={4cm 9.5cm 5cm 11.5cm},clip,width=0.95\textwidth]{figures/pics_phase/gnu_phase.pdf}
    \end{subfigure}
    \caption{Phase separation: Evolution of the energy $\widetilde{\mathcal{E}}$ and the mass conservation error for the density ratios $\rho_1:\rho_2\in\{10^0:10^3,10^0:10^2,10^0:10^1,10^1:10^0,10^2:10^0,10^3:10^0\}$. Energy difference between the symmetric results is of order $10^{-6}$.}
    \label{fig: meta}
    \end{figure}

\subsection{Convergence test}\label{subsec:conv}

For the convergence test we adapted the phase separation experiment for $\rho_1:\rho_2=1:100$ and consider 
\begin{equation}
   \phi_0(x,y)=0.2\sin(4\pi x)\sin(4\pi y), \qquad \vv_0(x,y) = 10^{-1}(\sin(\pi x)^2\sin(2\pi y),\sin(\pi y)^2\sin(2\pi x))^\top.
\end{equation}

The remaining parameters are unchanged.

\subsection*{Space Convergence:}
Since no exact solution is available we compute the error using a refined solution for fixed $\tau$. The error quantities we consider are
\begin{align*}
   \text{err}(\phi,h_k)&:= \max_{n\in\Itau}\norm{\phi^n_{{h_k}}-\phi^n_{h_{k+1}}}_{H^1(\Omega)}^2 , \qquad \text{err}(\vv,h_k):=\max_{n\in\Itau}\norm{\vv^n_{{h_k}}-\vv^n_{h_{k+1}}}_0^2 \\
   \text{err}(\mu+\alpha p,h_k)&:=\tau\sum_{n=1}^{n_T}\norm{(\mu^n_{h_k}+\alpha p^n_{h_k})-(\mu^n_{h_{k+1}}+\alpha p^n_{h_{k+1}})}_{H^1(\Omega)}^2 , \\
   \qquad \text{err}(\nabla\vv,h_k)&:=\tau\sum_{n=1}^{n_T}\norm{\vv_{h_k}^n-\vv^n_{h_{k+1}}}_{H^1(\Omega)}^2
\end{align*}

To this end we consider mesh refinements with the following mesh sizes $h_k \approx 2^{-1-k}$ for $k=0,\ldots,7.$  For the time step size we choose $\tau=10^{-3}$. We compute the experimental order of convergence (eoc) via $\textrm{eoc}_k=\log_2\left(\frac{\text{err}(a,h_{k-1})}{\text{err}(a,h_k)}\right)$ for the variables $a\in\{\phi,\u,\mu+\alpha p,\nabla\u\}.$ 

\begin{table}[htbp!]
	\centering
	\small
	\caption{$L^2(\Omega)$ errors and squared experimental order of convergence (eoc) up to time $T=0.1$} 
	\begin{tabular}{|c||c|c|c|c|c|c|c|c|}
		\hline
		$ k $ & $\text{err}(\phi,h_k)$ &  eoc & $\text{err}(\vv,h_k)$ & eoc & $\text{err}(\mu+\alpha p,h_k)$ &  eoc & $\text{err}(\nabla\vv,h_k)$ & eoc   \\
		\hline
        1 & $1.576\cdot 10^{-0}$    & --       & $1.984\cdot 10^{-3}$  & --       & $2.312\cdot 10^{-1}$ & --       & $2.164\cdot 10^{-2}$ & -- \\
        2 & $5.400\cdot 10^{-0}$    & $-$1.78  & $2.291\cdot 10^{-3}$  & $-$0.20  & $7.078\cdot 10^{-1}$ &$-$1.61   & $5.126\cdot 10^{-2}$ & $-$1.24\\
        3 & $2.718\cdot 10^{-0}$    & \phm0.99 & $1.117\cdot 10^{-3}$  & \phm1.04 & $5.699\cdot 10^{-1}$ & \phm0.31 & $7.496\cdot 10^{-2}$ & \phm0.55\\
        4 & $7.789\cdot 10^{-1} $   & \phm1.80 & $1.415\cdot 10^{-4}$  & \phm2.98 & $1.809\cdot 10^{-1}$ & \phm1.66 & $5.293\cdot 10^{-2}$ & \phm0.50\\
        5 & $1.986\cdot 10^{-1} $   & \phm1.97 & $9.139\cdot 10^{-6}$  & \phm3.95 & $4.619\cdot 10^{-2}$ & \phm1.97 & $1.467\cdot 10^{-2}$ & \phm1.85\\
        6 & $5.226\cdot 10^{-2} $   & \phm1.93 & $3.311\cdot 10^{-7}$  & \phm4.79 & $1.216\cdot 10^{-2}$ & \phm1.93 & $1.568\cdot 10^{-3}$ & \phm3.23\\
        7 & $1.330\cdot 10^{-2} $   & \phm1.97 & $1.503\cdot 10^{-8}$  & \phm4.46 & $3.104\cdot 10^{-3}$ & \phm1.97 & $9.934\cdot 10^{-5}$ & \phm3.98\\
		\hline
	\end{tabular}
	\label{table1}
\end{table}

Table \ref{table1} presents the errors and the (squared) experimental order of convergence in space. As expected, we observe first-order convergence for $(\phi,\mu+\alpha p)$ and second-order convergence for $\u$ (recall that the velocity is approximated by piecewise quadratic polynomials). The obtained rates are order optimal and the third order super convergence for the velocity in $L^\infty(0,T;L^2(\Omega))$ is not observed.

\subsection*{Time Convergence:}
Since no exact solution is available we compute the error using the reference solution at the finest space resolution for fixed $h$. We denote the solutions on different time resolutions by $(\phi_{\tau_k},\mu_{\tau_k},\u_{\tau_k},p_{\tau_k}).$ The error quantities we consider are
\begin{align*}
   \text{err}(\phi,\tau_k)&:= \max_{n\in\mathcal{I}_{\tau_k}}\norm{\phi^n_{\tau_k}-\phi^n_{\tau_{k+1}}}_{H^1(\Omega)}^2 , \qquad \text{err}(\u,\tau_k):=\max_{n\in\mathcal{I}_{\tau_k}}\norm{\vv^n_{{\tau_k}}-\vv^n_{\tau_{k+1}}}_0^2 \\
   \text{err}(\mu+\alpha p,\tau_k)&:=\tau_k\sum_{n=1}^{n_T}\norm{(\mu^n_{\tau_k}+\alpha p^n_{\tau_k})-(\bar\mu^n_{\tau_{k+1}}+\alpha \bar p^n_{\tau_{k+1}})}_{H^1(\Omega)}^2 , \\
   \qquad \text{err}(\nabla\vv,\tau_k)&:=\tau_k\sum_{n=1}^{n_T}\norm{\vv_{\tau_k}^n-\bar\vv^n_{\tau_{k+1}}}_{H^1(\Omega)}^2,
\end{align*}
where $\bar g_{\tau_{k+1}}^n:=\frac{1}{2}(g_{\tau_{k+1}}^{n}+g_{\tau_{k+1}}^{n-1/2}).$
To this end we consider time step refinements with the following sizes $\tau_k =10^{-4}\cdot2^{-k-1}$ for $k=0,\ldots,5.$  For the time step size we choose $h\approx 10^{-2}$. We compute the experimental order of convergence (eoc) via $\textrm{eoc}_k=\log_2\left(\frac{\text{err}(a,\tau_{k-1})}{\text{err}(a,\tau_k)}\right)$ for the variables $a\in\{\phi,\u,\mu+\alpha p,\nabla\u\}.$ 

\begin{table}[htbp!]
	\centering
	\small
	\caption{$L^2(\Omega)$ errors and squared experimental order of convergence (eoc) up to time $T=0.01$} 
	\begin{tabular}{|c||c|c|c|c|c|c|c|c|}
		\hline
		$ k $ & $\text{err}(\phi,\tau_k)$ &  eoc & $\text{err}(\u,\tau_k)$ & eoc & $\text{err}(\mu+\alpha p,\tau_k)$ &  eoc & $\text{err}(\nabla\u,\tau_k)$ & eoc   \\
		\hline
        1 & $1.558\cdot 10^{-7}$    & --   & $1.208\cdot 10^{-10}$  & --   & $1.145\cdot 10^{-9}$  & --   & $2.758\cdot 10^{-9}$   & -- \\
        2 & $4.653\cdot 10^{-8}$    & 1.74 & $3.122\cdot 10^{-11}$  & 1.95 & $2.919\cdot 10^{-10}$ & 1.97 & $7.696\cdot 10^{-10}$  & 1.84\\
        3 & $1.275\cdot 10^{-8}$    & 1.87 & $7.951\cdot 10^{-12}$  & 1.97 & $7.376\cdot 10^{-11}$ & 1.98 & $2.034\cdot 10^{-10}$  & 1.92\\
        4 & $3.339\cdot 10^{-9} $   & 1.93 & $2.007\cdot 10^{-12}$  & 1.99 & $1.855\cdot 10^{-11}$ & 1.99 & $5.231\cdot 10^{-11}$  & 1.96\\
        5 & $8.544\cdot 10^{-10} $  & 1.97 & $5.043\cdot 10^{-13}$  & 1.99 & $4.654\cdot 10^{-12}$ & 2.00 & $1.326\cdot 10^{-11}$ & 1.98\\
		\hline
	\end{tabular}
	\label{table2}
\end{table}

Table \ref{table2} presents the errors and the (squared) experimental order of convergence in time. As expected, we observe first-order convergence for all errors quantities. The obtained rates are order optimal.

\subsection{Rising bubble test cases}\label{subsec:risingbubble}
In this benchmark problem, a circular bubble of fluid 2 (the lighter phase) with an initial diameter of $D_0 = 2R_0 = 0.5$ is positioned at $(0.5, 0.5)$ within a rectangular domain $[0,1] \times [0,2]$, surrounded by fluid 1 (the heavier phase) \cite{hysing2009quantitative}. The initial distribution of the phase field is prescribed by:

\begin{align}\label{eq: init phi 2D}
  \phi^h_0(\mathbf{x}) = \tanh{\dfrac{\sqrt{(x-0.5)^2+(y-0.5)^2}-R_0}{\varepsilon\sqrt{2}}}.
\end{align}

Boundary conditions are applied as follows: no-penetration ($\mathbf{v} \cdot \mathbf{n} = 0$) on the vertical boundaries (left and right) and no-slip ($\mathbf{v} = 0$) on the horizontal boundaries (top and bottom). A schematic representation of the problem setup is provided in \cref{fig:sketch 2D rising bubble problem}.

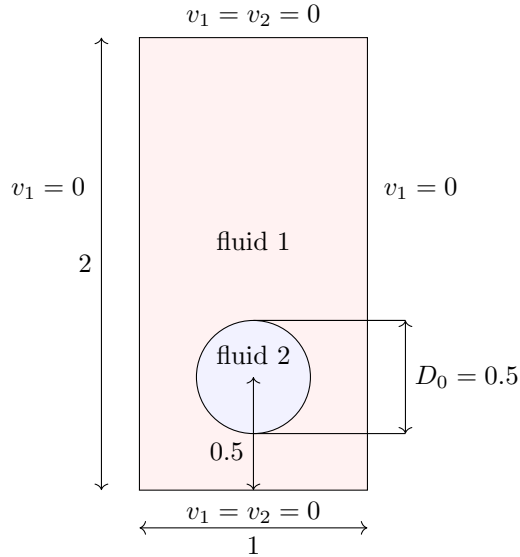
\begin{figure}[h]
\begin{center}
\begin{tikzpicture}
    \fill[fill=red!5,draw=black] (0, 0) rectangle (3, 6);
    
    \draw[fill=blue!5,draw=black] (1.5, 1.5) circle (0.75);
    
    \draw[<->] (0, -0.5) -- (3, -0.5) node[midway,below] {$1$};
    \draw[<->] (-0.5, 0) -- (-0.5, 6) node[midway,left] {$2$};
    \draw[<->] (3.5, 0.75) -- (3.5, 2.25) node[midway,right] {$D_0 = 0.5$};
    \draw[-] (1.5, 0.75) -- (3.5, 0.75);
    \draw[-] (1.5, 2.25) -- (3.5, 2.25);
    \draw[<->] (1.5, 0.0) -- (1.5, 1.5) node[midway,below left] {$0.5$};
    \draw[-] (1.5, 2.25) -- (3.5, 2.25);
    \node at (1.5, 6.3) {$v_1 = v_2 = 0$}; 
    \node at (1.5, -0.3) {$v_1 = v_2 = 0$};
    \node at (-1.2, 4.0) {$v_1 = 0$}; 
    \node at (3.7, 4.0) {$v_1 = 0$}; 
    \node at (1.5, 1.8) {fluid 2};
    \node at (1.5, 3.3) {fluid 1}; 
\end{tikzpicture}
    \caption{Rising bubble: Schematic representation of the problem setup}
    \label{fig:sketch 2D rising bubble problem}
\end{center}
\end{figure}

Simulations were conducted on a uniform rectangular mesh with element sizes $h = 1/32, 1/64, 1/128 $. The time step size is set as $\Delta t_n = 0.128 h $, while $\varepsilon = 0.64 h$. We choose $\gamma = \tilde{\sigma} \varepsilon$ and $\beta = \tilde{\sigma}/\varepsilon$ so that:
\begin{align}
    \Psi = \frac{\tilde{\sigma}}{4\epsilon}(1-\phi^2)^2 + \frac{\tilde{\sigma} \epsilon}{2}|\nabla\phi|^2, 
\end{align}
with $\tilde{\sigma} = 3 \sigma/(2\sqrt{2})$. Additionally, we select the mobility as $m = \bar{m} |1- (\phi_h^{n+1})^2|$ with $\bar{m} =0.1\varepsilon^2$. The term $\varepsilon^2$ is based on the scaling argument provided by Magaletti et al. \cite{magaletti2013sharp}. The benchmark problem consists of two cases, each defined by different parameter values, as listed in \cref{table: parameters 2D RB cases}. We refer to \ref{subsec:dim less} for the definitions of the dimensionless quantities.

\begin{table}[htbp]
\centering
\begin{tabularx}{\textwidth}{XXXXXXXXX}
Case & \hspace{0.1cm} $\rho_1$ & \hspace{0.1cm} $\rho_2$ & $\mu_1$ & $\mu_2$ & \hspace{0.1cm} $\sigma$ & \hspace{0.1cm} $g$ & $\mathbb{A}{\rm r}$ & $\mathbb{E}{\rm o}$ \\[4pt]
\hline\\[-6pt]
\hspace{0.5cm}1 & $1000$ & $100$ & $10$ & $1$   & $24.5$ & $0.98$ & $35$ & \hspace{0.05cm} $10$   \\[6pt]
\hspace{0.5cm}2 & $1000$ & \hspace{0.1cm} $1$   & \hspace{0.1cm} $1$  & $0.1$ &$1.96$  & $0.98$ & $35$ & $125$  \\[6pt]
\hline
\end{tabularx}
\caption{Parameters for the two-dimensional rising bubble cases.}
\label{table: parameters 2D RB cases}
\end{table}

\cref{fig: case 1 phi + contours,fig: case 2 phi + contours} display the zero level set of the phase-field for cases 1 and 2, respectively. In case 1, the bubble undergoes minimal deformation, whereas case 2 exhibits significant shape changes. In both cases, the solutions obtained on the two finest meshes are nearly indistinguishable. We verify conservation of the phase field and energy dissipation in \cref{fig: Mass,fig: Energy}, respectively.

\begin{figure}[!ht]
\captionsetup[subfigure]{justification=centering}
\begin{subfigure}{0.49\textwidth}
\centering
\includegraphics[height=1\textwidth]{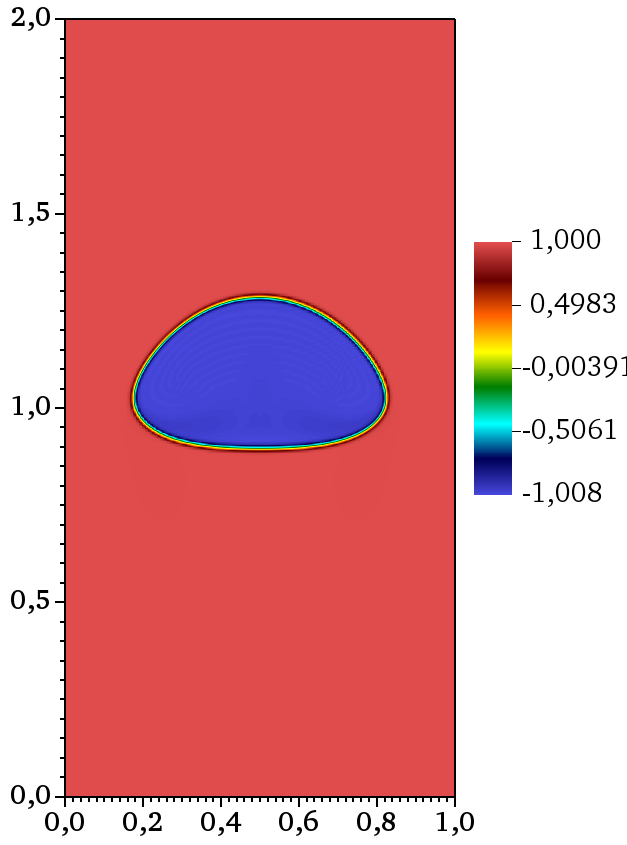}
\caption{$h = 1/128$\\{\color{white}.}}
\end{subfigure}
\begin{subfigure}{0.49\textwidth}
\centering
\includegraphics[height=1\textwidth]{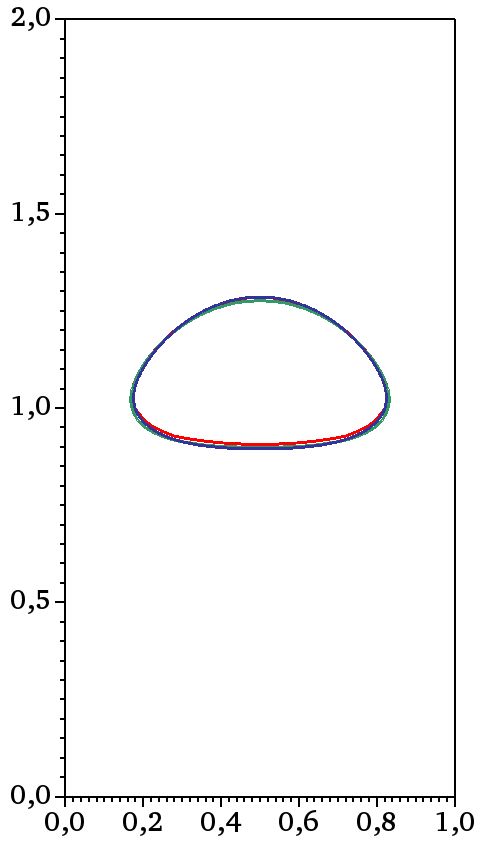}
\caption{$h = 1/32, 1/64, 1/128$\\  in red, green and blue (resp).}
\end{subfigure}
\caption{Case 1. Visualization of solution at final time $t=3$. (a) Phase field, (b) zero level set of the phase-field.}
\label{fig: case 1 phi + contours}
\end{figure}

\begin{figure}[!ht]
\captionsetup[subfigure]{justification=centering}
\begin{subfigure}{0.49\textwidth}
\centering
\includegraphics[height=1\textwidth]{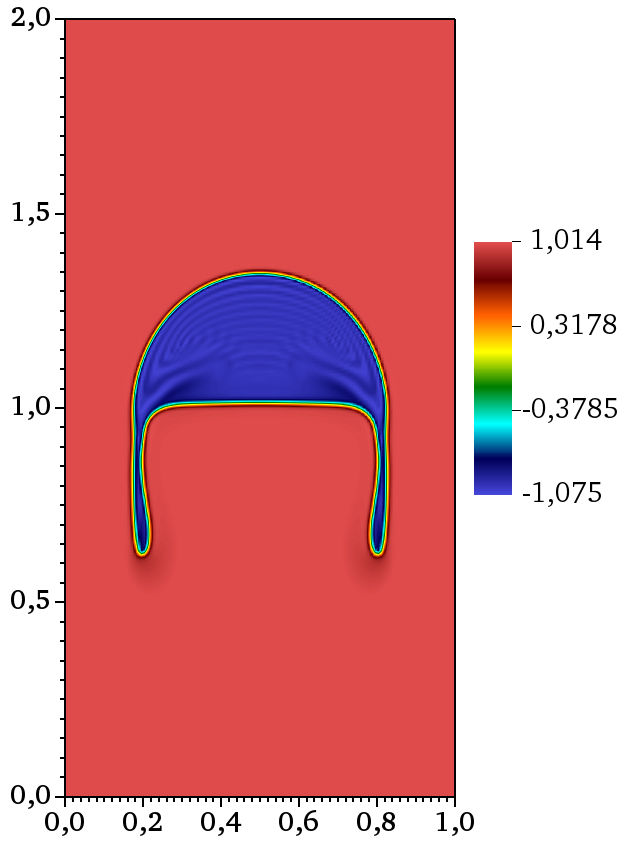}
\caption{$h = 1/128$\\{\color{white}.}}
\end{subfigure}
\begin{subfigure}{0.49\textwidth}
\centering
\includegraphics[height=1\textwidth]{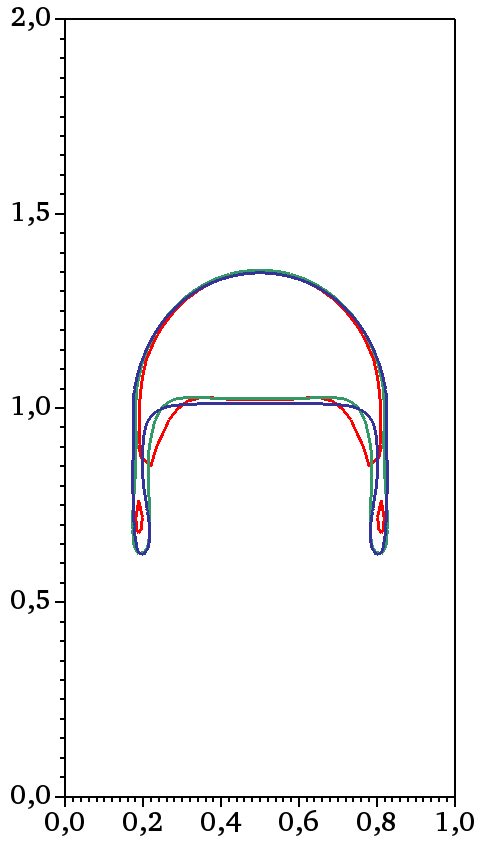}
\caption{$h = 1/32, 1/64, 1/128$\\  in red, green and blue (resp).}
\end{subfigure}
\caption{Case 2. Visualization of solution at final time $t=3$. (a) Phase field, (b) zero level set of the phase-field.}
\label{fig: case 2 phi + contours}
\end{figure}

\begin{figure}[!ht]
\begin{subfigure}{0.49\textwidth}
\centering
\includegraphics[width=0.95\textwidth]{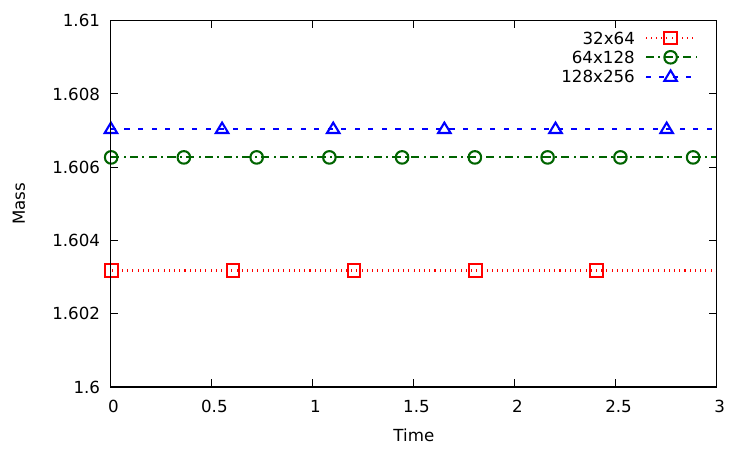}
\caption{Mass evolution Case 1.}
\end{subfigure}
\begin{subfigure}{0.49\textwidth}
\centering
\includegraphics[width=0.95\textwidth]{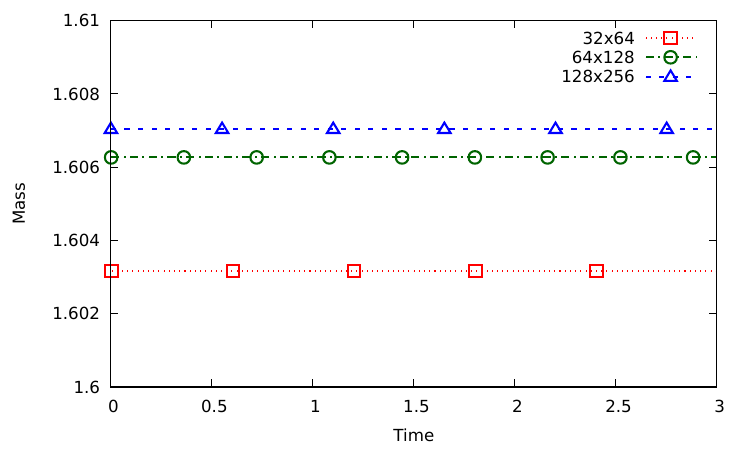}
\caption{Mass evolution Case 2.}
\end{subfigure}
\centering
\begin{subfigure}{0.49\textwidth}
\centering
\includegraphics[width=0.95\textwidth]{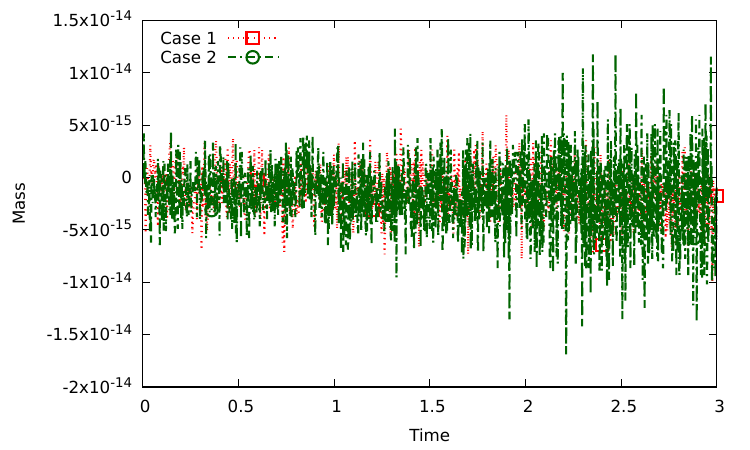}
\caption{Mass evolution error. Cases 1 and 2.}
\end{subfigure}
\caption{Evolution of the mass conservation (error); (a) and (b) for $h = 1/32, 1/64, 1/128$, (c) Zoom for $h = 1/128$.}
\label{fig: Mass}
\end{figure}

\begin{figure}[!ht]
\begin{subfigure}{0.49\textwidth}
\centering
\includegraphics[width=0.95\textwidth]{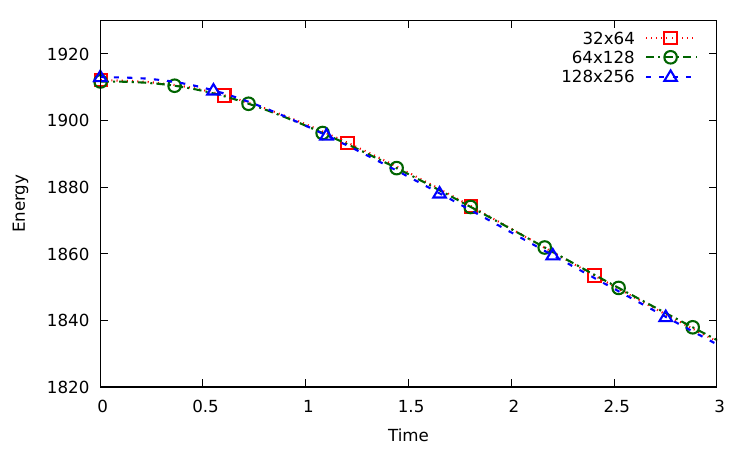}
\caption{Case 1.}
\end{subfigure}
\begin{subfigure}{0.49\textwidth}
\centering
\includegraphics[width=0.95\textwidth]{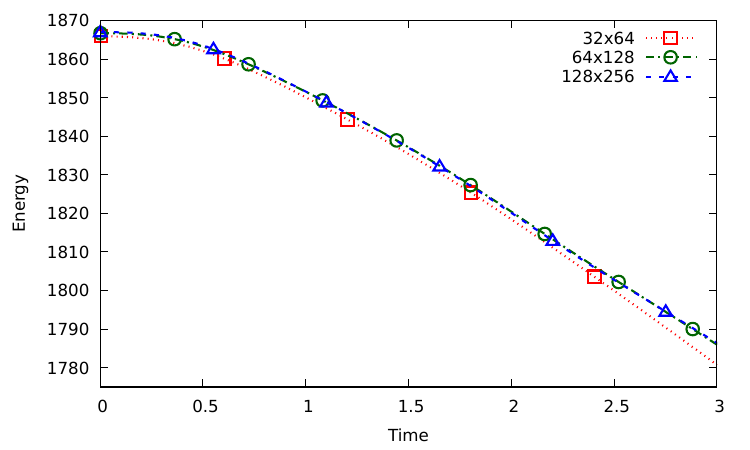}
\caption{Case 2.}
\end{subfigure}
\caption{Evolution of the energy $E$ for different mesh widths $h = 1/32, 1/64, 1/128$.}
\label{fig: Energy}
\end{figure}

Next, to facilitate a quantitative comparison with reference results from the literature, we consider two key metrics: the center of mass ($y_b$) and the rise velocity ($v_b$), defined as:
    \begin{align}
        y_b := \dfrac{\int_{\phi < 0} y ~{\rm d}x}{\int_{\phi < 0} ~ {\rm d}x},\qquad
        v_b := \dfrac{\int_{\phi < 0} v_2 ~{\rm d}x}{\int_{\phi < 0} ~ {\rm d}x}.
    \end{align}

In \cref{fig: case 1 CoM,fig: case 1 RV,fig: case 2 CoM,fig: case 2 RV}, we present the center of mass and rise velocity for both test cases, considering different mesh sizes. These results are compared with computational data from the literature, including simulations performed using the TP2D, FreeLIFE, and MooNMD codes \cite{hysing2009quantitative}, as well as the NSCH models proposed by Abels et al. \cite{abels2012thermodynamically}, Boyer \cite{boyer2002theoretical}, and Ding et al. \cite{ding2007diffuse}. These NSCH computations were carried out by Aland and Voigt \cite{aland2012benchmark}. In addition, we compare with the volume-averaged velocity formulation of the NSCH model \cite{ten2024divergence}. In the figures we denote these as `$\text{NSCH}_{\text{vol}}$' and we denote the current computations by `$\text{NSCH}_{\text{mass}}$'.

For both cases, the center of mass shows good agreement with the reference data. However, in case 2, significant deviations in the rise velocity are observed for $t>1.5$. In this regime, our results align closely with the NSCH computations of Aland and Voigt \cite{aland2012benchmark}, but differ from those obtained using the TP2D, FreeLIFE, and MooNMD codes.

\begin{figure}[!ht]
\begin{subfigure}{0.49\textwidth}
\centering
\includegraphics[width=0.95\textwidth]{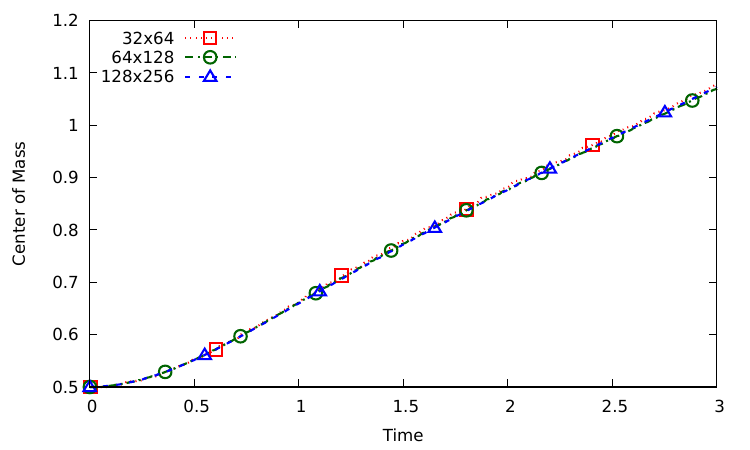}
\caption{$h = 1/32, 1/64, 1/128$.}
\end{subfigure}
\begin{subfigure}{0.49\textwidth}
\centering
\includegraphics[width=0.95\textwidth]{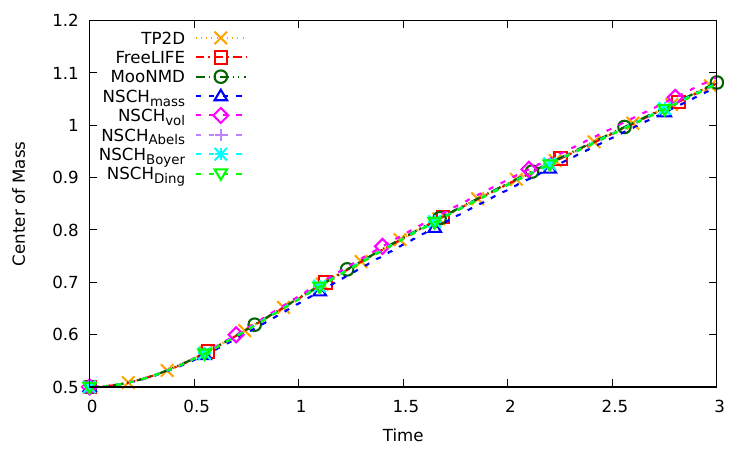}
\caption{Comparison with data from the literature.}
\end{subfigure}
\caption{Case 1. Center of mass (a) for different mesh sizes, and (b) a comparison of the finest mesh results to reference data.}
\label{fig: case 1 CoM}
\end{figure}

\begin{figure}[!ht]
\begin{subfigure}{0.49\textwidth}
\centering
\includegraphics[width=0.95\textwidth]{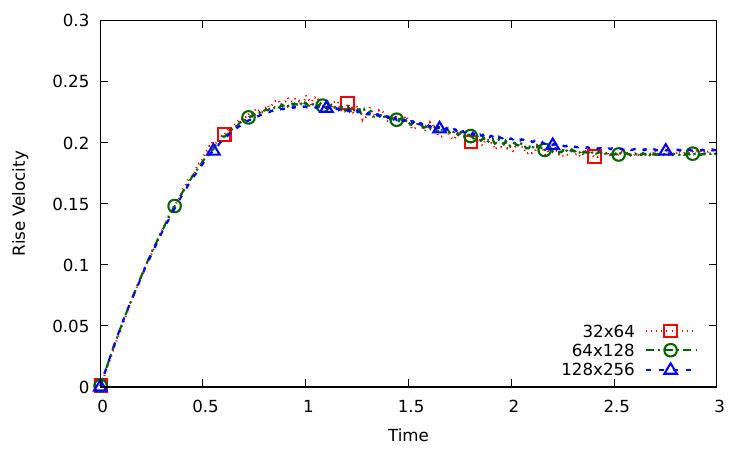}
\caption{$h = 1/32, 1/64, 1/128$.}
\end{subfigure}
\begin{subfigure}{0.49\textwidth}
\centering
\includegraphics[width=0.95\textwidth]{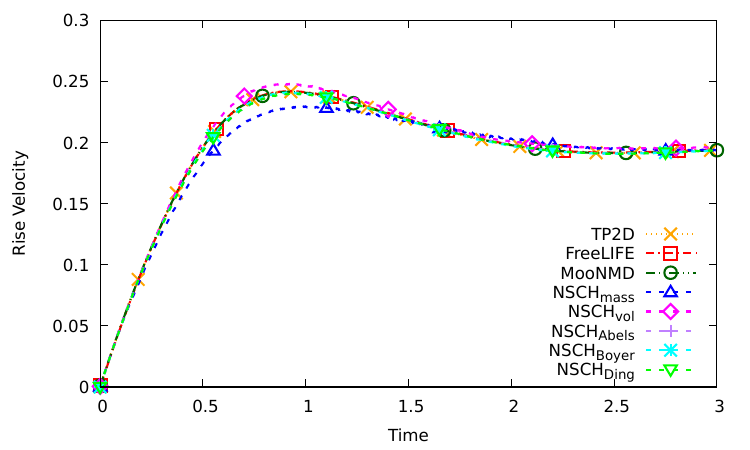}
\caption{Comparison with data from the literature.}
\end{subfigure}
\caption{Case 1. Rise velocity (a) for different mesh sizes, and (b) a comparison of the finest mesh results to reference data.}
\label{fig: case 1 RV}
\end{figure}

\begin{figure}[!ht]
\begin{subfigure}{0.49\textwidth}
\centering
\includegraphics[width=0.95\textwidth]{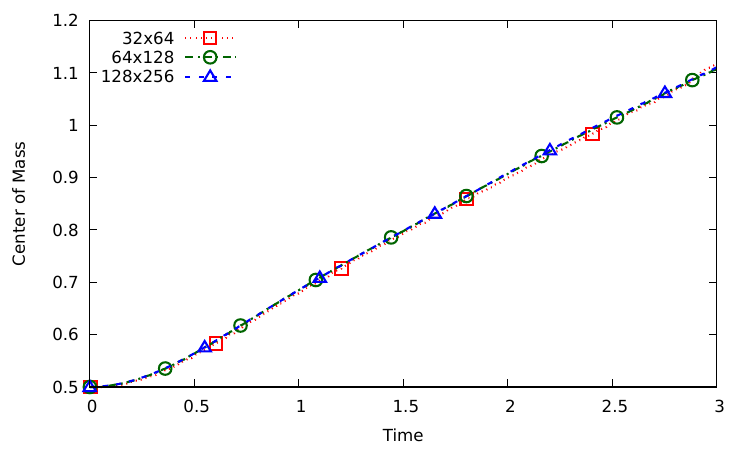}
\caption{$h = 1/16, 1/32, 1/64, 1/128$.}
\end{subfigure}
\begin{subfigure}{0.49\textwidth}
\centering
\includegraphics[width=0.95\textwidth]{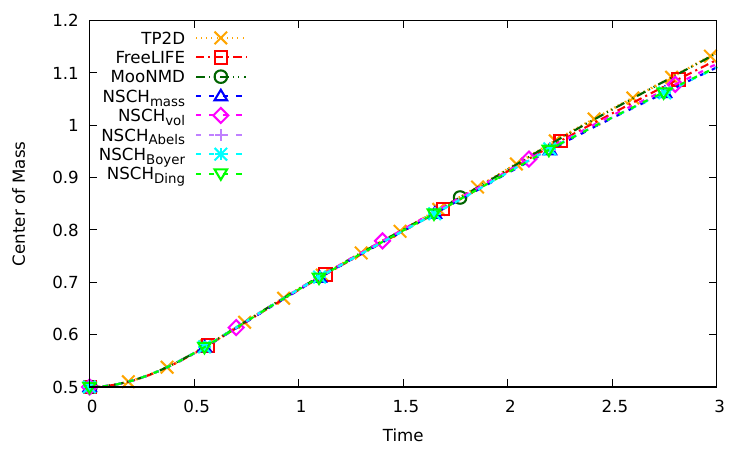}
\caption{Comparison with data from the literature.}
\end{subfigure}
\caption{Case 2. Center of mass (a) for different mesh sizes, and (b) a comparison of the finest mesh results to reference data.}
\label{fig: case 2 CoM}
\end{figure}

\begin{figure}[!ht]
\begin{subfigure}{0.49\textwidth}
\centering
\includegraphics[width=0.95\textwidth]{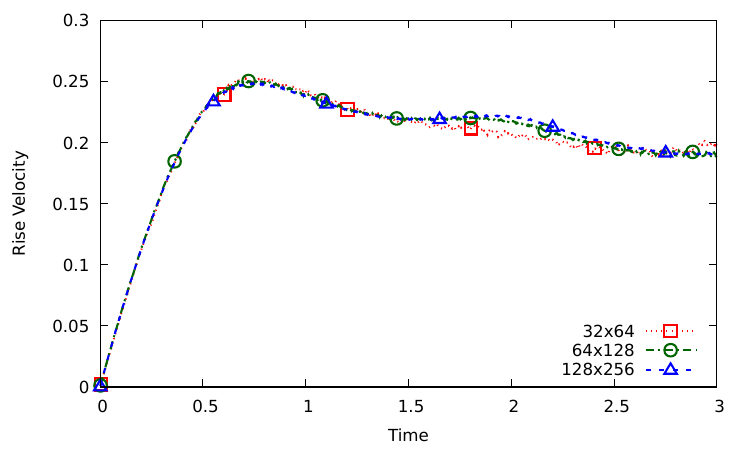}
\caption{$h = 1/32, 1/64, 1/128$.}
\end{subfigure}
\begin{subfigure}{0.49\textwidth}
\centering
\includegraphics[width=0.95\textwidth]{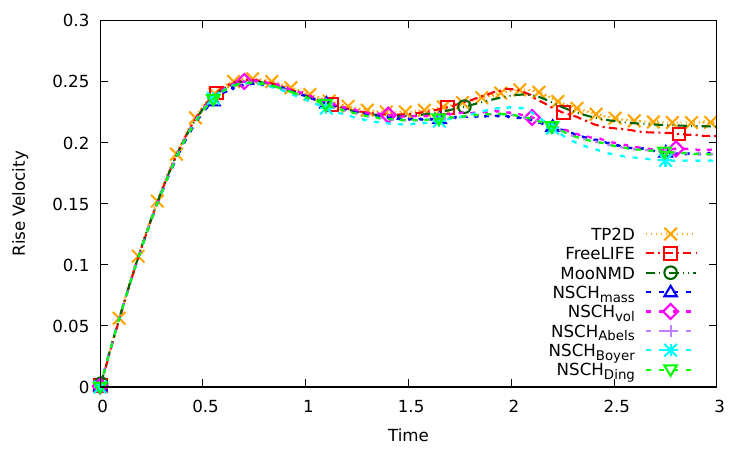}
\caption{Comparison with data from the literature.}
\end{subfigure}
\caption{Case 2. Rise velocity (a) for different mesh sizes, and (b) a comparison of the finest mesh results to reference data.}
\label{fig: case 2 RV}
\end{figure}
\newpage
\section{Summary and outlook}\label{sec:conclusion and outlook}
In this work, we developed a fully-discrete, unconditionally, energy-stable method for the two-phase Navier-Stokes Cahn-Hilliard (NSCH) mixture model with non-matching densities. By incorporating positive extensions of the density directly into the kinetic energy, we ensured provable stability, addressing a key challenge in the numerical discretization of NSCH models, particularly for large density ratios. The proposed method is based on an equivalent reformulation of the NSCH model using mass-averaged velocity and volume fraction-based order parameters, which simplifies implementation while maintaining theoretical consistency. Numerical results confirm the accuracy, robustness, and stability of the scheme in capturing complex two-phase flow dynamics.

While this work provides a consistent and computationally efficient framework for two-phase flows, several directions remain for future research. We mention three of these. First, extending the method to N-phase flows \cite{eikelder2024unified,eikelder2025compressible} would allow for the simulation of more complex multiphase systems, requiring appropriate generalizations of the energy-stable formulation. Second, alternative free-energy potentials, such as the Flory-Huggins model, could be incorporated to better capture thermodynamically consistent phase separation phenomena. Third, the computation of high-Reynolds number flows constitute an important research direction. This demand the development of novel (thermodynamically-consistent) stabilized methods, such as \cite{ten2018correct}.

\section*{Acknowledgments}
The authors wish to thank Sebastian Aland for sharing the simulation data of the Navier-Stokes Cahn-Hilliard models. A. Brunk~gratefully acknowledges the support of the German Science Foundation via TRR~146 -- project number 233630050 --, SPP2256 ``Variational Methods for Predicting Complex Phenomena in Engineering Structures and Materials -- project number 441153493 --, and the support by the Gutenberg Research College, JGU Mainz.

\appendix

\section{Non-dimensional formulation}\label{subsec:dim less}

We re-scale the system  based on the following dimensionless variables:
  \begin{align}\label{eq: ref values}
    \bx^* =&~ \frac{\bx}{X_0}, \quad t^* = \frac{t}{T_0}, \quad \vv^* = \frac{\vv}{V_0}, \quad \rho^* = \frac{\rho}{\rho_1},  \quad \eta^* = \frac{\eta}{\eta_1}, \nn\\
    p^* =&~ \frac{p}{\rho_1 V_0^2}, \quad \mu^* = \frac{\mu}{\rho_1 V_0^2}, \quad m^* = m \rho_1 V_0,
  \end{align}
  where $X_0$ is a characteristic length scale, $T_0$ is a characteristic time scale, and $V_0=X_0/T_0$ is a characteristic velocity. The dimensionless system reads:
\begin{subequations}\label{eq: model orig dim less}
  \begin{align}
   \partial_t (\rho \vv) + \div \left( \rho \vv\otimes \vv \right) + \nabla p + \phi \nabla \mu  & \nn\\
    - \frac{1}{\mathbb{R}{\rm e}}\div \left(   \eta (2\mathbf{D}+\lambda({\rm div}\vv) \mathbf{I}) \right)+ \frac{1}{\mathbb{F}{\rm r}^2} \rho\boldsymbol{\jmath} &=~ 0, \label{eq: model orig: mom dim less}\\
   \div (\vv) - \alpha \mathbb{C}{\rm n}^2\div \left(m \nabla (\mu+\alpha p)\right) &=~0, \label{eq: model orig: cont dim less} \\
  \partial_t \phi +   \div (\vv\phi) -  \mathbb{C}{\rm n}^2\div \left(m \nabla (\mu+\alpha p)\right) &=~0,\label{eq: model orig: PF dim less}\\
  \mu - \frac{1}{\mathbb{W}{\rm e}\mathbb{C}{\rm n}} f'(\phi)+\frac{\mathbb{C}{\rm n}}{\mathbb{W}{\rm e}} \Delta \phi&=~0,
  \end{align}\label{eq: model orig: mu dim less}
\end{subequations}
where we have omitted the $*$ symbols. The dimensionless coefficients are the Reynolds number ($\mathbb{R}{\rm e}$), the Weber number ($\mathbb{W}{\rm e}$), the Froude number ($\mathbb{F}{\rm r}$) and the Cahn number ($\mathbb{C}{\rm n}$) given by:
\begin{subequations}\label{eq: dimensionless quantities}
\begin{align}
     \mathbb{R}{\rm e} =~ \frac{\rho_1 V_0 X_0}{\eta_1},\quad\mathbb{W}{\rm e} =~ \frac{\rho_1 V_0^2 X_0}{\sigma},\quad\mathbb{F}{\rm r} =~ \frac{V_0}{\sqrt{g X_0}},\quad \mathbb{C}{\rm n} =~ \frac{\varepsilon}{ X_0}. 
\end{align}
\end{subequations}
The associated energies density take the form:
\begin{align}
    K =~ \tfrac{1}{2}\rho \vv\cdot \vv, \qquad G =~ \frac{1}{\mathbb{F}{\rm r}^2}\rho y, \qquad \Psi =~ \frac{\mathbb{C}{\rm n}}{2\mathbb{W}{\rm e}}  \nabla \phi\cdot\nabla \phi + \frac{1}{\mathbb{W}{\rm e}\mathbb{C}{\rm n}} f(\phi).
\end{align}

The dynamics of rising bubble problems are typically characterized by the E\"{o}tv\"{o}s number ($\mathbb{E}{\rm o}$) and the Archimedes number ($\mathbb{A}{\rm r}$):
\begin{subequations}\label{eq: dimensionless quantities 2}
\begin{align}
     \mathbb{E}{\rm o} =&~ \frac{\rho_1 g D_0^2}{\sigma},\\
     \mathbb{A}{\rm r} =&~ \dfrac{\rho_1\sqrt{g D_0^3}}{\nu_1},
\end{align}
\end{subequations}
where $D_0$ represents the diameter of a bubble. The E\"{o}tv\"{o}s number, also referred to as the Bond number, quantifies the balance between gravitational and surface tension forces. Meanwhile, the Archimedes number represents the ratio of buoyancy to viscous forces, capturing the influence of fluid inertia in the system. By choosing
    \begin{align}
        X_0 = D_0, \qquad U_0 = \dfrac{D_0}{T_0}, \qquad T_0 = \sqrt{\dfrac{\rho_1 D_0^3}{\sigma}},
    \end{align}
the dimensionless numbers are related via:
    \begin{align}
        \mathbb{R}{\rm e} = \mathbb{E}{\rm o}^{-1/2}\mathbb{A}{\rm r}, \qquad \mathbb{F}{\rm r} = \mathbb{E}{\rm o}^{-1/2}, \qquad 
        \mathbb{W}{\rm e} = 1,
    \end{align}
where $T_0$ is the capilary time scale.

\bibliographystyle{unsrt}
\bibliography{references}

\end{document}